\documentclass[a4paper]{amsart} 
\usepackage[latin1]{inputenc}
\usepackage{amsmath,amsthm,amsfonts,amssymb,amstext,amsfonts,amscd,
xcolor}
\usepackage[english]{babel}
\usepackage{pxfonts} 
\usepackage{mathrsfs}
\usepackage{palatino} 
\usepackage{graphicx}
\numberwithin{equation}{section}

\DeclareMathAlphabet{\mathpzc}{OT1}{pzc}{m}{it}
\def \bigchi {\mbox{\Large$\chi$}}
\def \al {\alpha } \def \be {\beta } \def \de {\delta }
\def \te {\theta} \def \fhi {\varphi} \def \ep {\epsilon} 
\def \ga {\gamma} \def \la {\lambda} 
  \def\vfi{\varphi}
\def \constantB {\mbox{\large$\varsigmaup$}} 
\def \secondelta {\sigma} 
\def \bigchi {\mbox{\Large$\chi$}} 
\def \crit {\mathscr{C}} 

\def \Disc {\mathcal{D}} 
\def \disc {\mathscr{D}} 
\def \cc {\mathscr{S}} 
\def \B {\mathcal{B}}
\def \H {\mathcal{H}}

\def \cL {\mathcal{L}}
\def \ra {\rightarrow }

\def \ilim {\liminf_{n\ra\infty}\, }
 
\def \soma {\sum_{i=1}^{n}} 

\def \s {\textquoteright}

\newcommand{\prob}{\operatorname{p}}
\newcommand{\leb}{\operatorname{m}}
\newcommand{\Leb}{\operatorname{Leb}}

\newcommand{\dist}{\operatorname{dist}}

\newcommand{\Y}{\ensuremath{\mathbb{Y}}}

\newcommand{\R}{\ensuremath{\mathbb{R}}}

\newcommand{\circulo}{\ensuremath{\mathbb{S}}}
\newcommand{\toro}{\ensuremath{\mathbb{X}}}
\newcommand{\N}{\ensuremath{\mathbb{N}}}
\newcommand{\Z}{\ensuremath{\mathbb{Z}}}
\newcommand{\T}{\ensuremath{\mathbb{T}}}

\newcommand{\p}{^\prime}

\newcommand{\sumaupla}[1]{a_1+a_2+\ldots+a_#1}

\newcommand{\measure}{\varpi}

\theoremstyle{plain}

\newtheorem{maintheorem}{Theorem}
\newtheorem{maincorollary}[maintheorem]{Corollary}

\newtheorem{theorem}{Theorem}[section]
\newtheorem{proposition}[theorem]{Proposition}
\newtheorem{corollary}[theorem]{Corollary}
\newtheorem{lemma}[theorem]{Lemma}
\newtheorem{claim}[theorem]{Claim}

\theoremstyle{definition}
\newtheorem{example}{Example}
\newtheorem{remark}[theorem]{Remark}
\newtheorem{definition}[theorem]{Definition}

\begin{thanks} {V.A. and J.S. were partially supported by
    CNPq, FAPERJ, FAPESB and PRONEX (Brazil).}
\end{thanks}

\title[ACIM for random non-uniformly expanding maps]{Absolutely
continuous invariant measures for random non-uniformly expanding maps}

\date{\today}

\author{Vitor Araujo}

\address{V. Araujo, 
Departamento de Matem\'atica, Universidade Federal da Bahia
\\  Av. Ademar de Barros s/n, 40170-110 Salvador, Brazil.}

\email{vitor.d.araujo@ufba.br and vitor.araujo.im.ufba@gmail.com}

\author{Javier Solano}

\address{Javier Solano, Instituto de Matem\'atica e Estat\'istica,
  Universidade Federal Fluminense, Rua M\'ario Santos Braga,
  s/n, Valonguinho, 24.020-140 Niter\'oi, RJ-Brazil}
\email{jsolano@impa.br}

\subjclass[2010]{Primary:
37D25. 
Secondary:37E05, 37HXX
}

\begin{document}

\begin{abstract}
  We prove existence of (at most denumerable many)
  absolutely continuous invariant probability measures for
  random one-dimensional dynamical systems with asymptotic
  expansion.  If the rate of expansion (Lyapunov exponents)
  is bounded away from zero, we obtain finitely many ergodic
  absolutely continuous invariant probability measures,
  describing the asymptotics of almost every point. We also
  prove a similar result for higher-dimensional random
  non-uniformly expanding dynamical systems.
  The results are consequences of the
    construction of such measures for skew-products with
    essentially arbitrary base dynamics and asymptotic
    expansion along the fibers.  In both cases our method
  deals with either critical or singular points for the
  random maps.
\end{abstract}

\maketitle


\section{Introduction}
\label{sec:introd}

In this work we study the existence of absolutely continuous
invariant probability measures for the random iteration of
maps of the interval, or of a compact manifold, which have
positive Lyapunov exponents but can also have critical
points or singularities. We also obtain a decomposition of
each absolutely continuous invariant measure into at most
denumerably many absolutely continuous ergodic components.
This can be seen as an extension of the results of
Pelikan~\cite{Pelik84}, Morita~\cite{Mor85} and
Buzzi~\cite{Buzz00} which deal with random iterations of
piecewise expanding maps.

It is well-known that the dynamics of random maps can be
modeled by a skew-product map where the ``noise'' is driven
by the ergodic base transformation. This is the general form
of a Random Dynamical System; see \cite[Definition
1.1.1]{arnold-l-1998}. Hence our results can also be seen as
a study of the dynamics of skew-product whose maps along the
one-dimensional fibers have critical points or
discontinuities, positive Lyapunov exponents and very weak
conditions on the base transformation. We mention the work
of Denker and Gordin \cite{DenkGord99} together with
Heinemann \cite{DenkGordHein02} where equilibrium states for
random bundle dynamics were studied under the assumption of
expansion along the fibers.

As an example of application of our results, let us consider
the map $\vfi(\theta,x)=(\alpha(\theta),f(\theta,x))$ with
$\alpha:\circulo^1\to\circulo^1$ a continuous map with an
ergodic $\alpha$-invariant probability measure $\nu$; and
$f_\theta(x)=a(\te)-x^2$ for $a(\te)$ continuous so that
$\vfi$ is well-defined, and $\leb$ the Lebesgue measure on
the interval $[-2,2]$. We use the notation
$\fhi^n(\te,x)=(\al^n(\te), f_\te^n(x))$, where we write
$\theta_n=\alpha^n(\theta), n\ge0$ and
$f^n_\theta(x)=(f_{\theta_{n-1}}\circ\dots\circ
f_{\theta_0})(x)$, which can be regarded as the random
composition of maps from the family $f_\theta$ chosen
according to the measure preserving transformation $\alpha$.

\begin{corollary}
  \label{cor:Viana-noexpansion}
  Assume that there exists $\la>0$ such that, for
$\nu\times\leb$-a.e. $(\te,x)$,
  \begin{equation} \label{eq:posLyap}
\ilim \frac{1}{n} \log |Df_\te^n(x)|\geq\la
\end{equation}
  Then $\vfi$ admits finitely many ergodic invariant
  probability measures absolutely continuous with respect to
  $\nu\times \leb$. Moreover, $\nu\times\leb$-a.e.
$(\te,x)$ belongs to the basin of one of these measures.
\end{corollary}

The weak assumptions of the dynamics of the base map allows
us to state our results in the setting of random dynamical
systems; see Corollary \ref{cor:finitelySRBquadratic} in
Subsection \ref{sec:random-dynamic-resul} for details. The
assumption (\ref{eq:posiexponents}) is natural if we
consider random perturbations of certain non-uniformly
expanding maps which are stochastically stable. Namely, if
$f_{\bar\theta}$ is a non-uniformly expanding $C^2$ local
diffeomorphism of a compact manifold $Y$, $f_\theta$ is a
$C^2$ family of maps and $\alpha:\toro\to\toro$ is the left
shift map on the infinite product
$\toro=[\theta-\epsilon,\theta+\epsilon]^\N$ such that
$f_{\bar\theta}$ is stochastically stable, then the
skew-product map $\vfi(\theta,x)$ satisfies
(\ref{eq:posiexponents}); see \cite[Theorem B]{AA03} and
compare with Examples~\ref{ex:intermittent}
and~\ref{ex:higher-dim} in Section~\ref{sec:exampl} together
with the stochastically stable examples from
\cite{ArTah}. For other families of non-uniformly expanding
maps, even when stochastic stability is known, non-uniform
expansion for random orbits is an interesting open question;
see Section~\ref{sec:exampl}.  An open set of maps
satisfying (\ref{eq:posiexponents}) is provided by Viana in
\cite{Vi97}, see below. On the other hand, our work can be
useful not only for the study of small random perturbations
of a given dynamical system. In our results, the maps
$f_{\theta_0}, f_{\theta_1}, \dots$ are not given
necessarily by an i.i.d. process and they can be distant
from each other. Our results hold for general random
dynamical systems on the interval which are non-uniformly
expanding, see subsections \ref{sec:random-dynamic-resul}
and \ref{sec:one-dimens-fibers} for the precise setting of
the work.

%

This work can also be seen as a generalization
of the earlier work of Keller~\cite{Ke90} which proves that
\emph{for maps of the interval with finitely many critical
  points and non-positive Schwarzian derivative, existence
  of absolutely continuous invariant probability is
  guaranteed by positive Lyapunov exponents, i.e.,}
\begin{align*}
  \limsup_{n\to+\infty} \frac1n \log |Df^n (x)| > 0
  \quad\text{on a positive
  measure set of points $x$}.
\end{align*}

Related results were obtained by Alves, Bonatti and
Viana~\cite{ABV00}. They show that every non-uniformly expanding local
diffeomorphism away from a non-degenerate critical/singular
set, on any compact manifold, admits a finite number of
ergodic absolutely continuous invariant measures describing
the asymptotics of almost every point. The notion of
non-uniform expansion means that
\begin{align}\label{eq:NUEdef}
  \liminf_{n\to+\infty}\frac1n\sum_{j=0}^{n-1} \log\|
  Df(f^j(x))^{-1}\| <0 \quad \text{Lebesgue almost everywhere}.
\end{align}
Some control of recurrence to
this critical/singular set must be assumed to construct the
absolutely continuous invariant measures. This assumption is
usually rather difficult to verify.


The main known example of maps satisfying the conditions of
the result of Alves, Bonatti and Viana are the \emph{Viana
  maps}. These maps were introduced by Viana \cite{Vi97} and
studied by many authors,
e.g. \cite{Al00,AA03,alves-viana2002,buzzi-sester-tsujii,schnell08}
among others. The maps are skew-products $\vfi:\toro\times
\Y\to\toro\times \Y, (\theta,x)\mapsto
(\alpha(\theta),f(\theta,x))$, with $\al$ being a uniformly
expanding circle map and the maps on the fibers being
quadratic maps of the interval $f_\theta(x)=a(\te)-x^2$ for
$a(\te)=a_0+\beta\sin(2\pi\theta), \beta>0$ small and $a_0$
a Misiurewicz parameter for $f_{a_0}$. The central direction
along $\Y$ is dominated by the strong expansion of the base
dynamics along $\toro$.  For an open class of these maps,
Viana \cite{Vi97} proved the positiveness of the Lyapunov
exponents and Alves \cite{Al00} proved the existence of an
absolutely continuous invariant measure.


Extensions of the above mentioned results were obtained,
  among others, by Pinheiro~\cite{Pinheiro05}, and by one of
  the authors \cite{solano} but, in all cases, either
  non-uniform expansion (\ref{eq:NUEdef}) in all directions,
  or a weaker form of hyperbolicity (partial hyperbolicity)
  is demanded. The critical/singular set
is also assumed to be non-degenerate.
  In a remarkable work, Tsujii~\cite{Tsu05}
  proves results in this line for generic partially
  hyperbolic endomorphisms on compact surfaces.  
  
  On the other hand, for piecewise expanding maps in higher
  dimensions, the existence of absolutely continuous
  invariant measures was obtained by
  Adl-Zarabi~\cite{Adl96}, Buzzi~\cite{Bu00}, Gora-Boyarsky
  \cite{GorBoy89}, Keller \cite{Ke79} and, among other,
  Saussol \cite{Sa00}. Again the authors assume uniform
  expansion with strong expansion rates together with
  certain boundary conditions on the pieces of the domain
  where the transformation is not expanding.

  Our results demand no partial hyperbolicity or domination
  conditions and we put no restriction on the dynamics of
  the base of the skew-product, other than almost everywhere
  continuity and the existence of an invariant ergodic
  probability measure. We do not require non-uniform
  expansion (\ref{eq:NUEdef}) in all directions, nor the
  non-degenerate conditions of the critical set.  Along
  multidimensional fibers (i.e. the dimension of the space
  $\Y$), we do demand non-uniform expansion and a control of
  the recurrence to the singular/critical set. Along
  one-dimensional fibers (i.e., the case where $\Y$ is the
  interval) with $f_\theta$ having negative Schwarzian, we
  assume non-uniform expansion only: we do not assume slow
  recurrence. In particular, the base dynamics can have no
  absolutely continuous invariant measure with respect to
  some natural volume form, as we present in some
  examples. Under these mild conditions we prove the
  existence of at most denumerable many invariant
  probability measures absolutely continuous along the
  fibers. We get finitely many invariant
  probability measures, instead of denumarable many, if the 
  rate of non-uniform expansion is bounded away from zero. For
non-uniformly expanding random
  dynamical systems on the interval, we get finitely many 
  absolutely continuous measures defined on the interval, 
  describing the asymptotics of almost all random orbits.

\subsection{Statements of results}
\label{sec:statem-results}

For a topological space $X$ we denote by
$\mathcal{B}_{X}$ the Borel $\sigma$-algebra on $X$.
The main setting is the following: let $\toro$ and $\Y$ be a
separable metrizable and complete (i.e., Polish) topological spaces.
Let us consider the skew-product map
\begin{equation*}
\begin{array}{lccc}
\fhi:&\toro\times \Y&\longrightarrow& \toro\times \Y \\
&(\te,x)&\mapsto& (\al(\te), f(\te,x)).
\end{array}
\end{equation*}
We assume that $\vfi$ is at least measurable with respect to the
Borel $\sigma$-algebra $\mathcal{B}_{\toro}\times
\mathcal{B}_{\Y}$ (which equals $\mathcal{B}_{\toro\times\Y}$ since
both $\toro$ and $\Y$ are separable metric
spaces; see e.g.~\cite[Appendix M.10]{billingsley99}).


\subsubsection{One dimensional fibers}
\label{sec:one-dimens-fibers}
We consider $\Y=I_0$ a compact interval. For $\te\in\toro$,
$f_\te:I_0\ra I_0$, $x\ra f(\te,x)$ is an
interval map, possibly with critical points and
discontinuities. We denote by $\crit_\te$ and $\disc_\te$
the set of critical points and discontinuities,
respectively, of $f_\te$, for every $\te\in\toro$. {We also
  use the notations $\crit=\{ (\te,x)\in \toro\times I_0 ;
  x\in \crit_\te\}$ and $\disc=\{ (\te,x)\in \toro\times I_0
  ; x\in \disc_\te \}$}.

We assume throughout that
the discontinuities $\disc_\theta$ of the interval map
$f_\theta$ are in the interior of $I_0$, and
  that the lateral limits exist at each $x\in\disc_\theta$;
  see condition $(H_4^*)$ in what follows.

We assume also that
\begin{enumerate}
\item[($H_1$)] $p:=\sup \{\# (\crit_\te\cup\disc_\te),
  \te\in\toro \} <\infty$ and $\Gamma:=\sup \{ \partial_x
  f(\te,x) , (\te,x)\notin \disc_{\te} \} <\infty $. The set
  \begin{equation*}
 {\cc=\{ (\te,x)\in \toro\times I_0 ; x\in \crit_\te\cup\disc_\te \}}
\end{equation*}
is measurable (i.e. it belongs to $\mathcal{B}_{\toro}\times
\mathcal{B}_{I_0}$).
\item[$(H_2)$] $\al:\toro\ra\toro$ is a measurable map with
  an ergodic invariant probability measure $\nu$ such that
  $\nu(\Disc_\alpha)=0$, where
  $\Disc_\alpha$ is the set of discontinuity points of
  $\alpha$.
\end{enumerate}
The assumption on the discontinuity set is a natural
condition to study the $\vfi$-invariance of weak$^*$
accumulation points of dynamically defined probability measures.
Let us consider the map
\begin{equation*}
F:\toro\to B(I_0) \quad
\theta\mapsto f_\te:I_0\to I_0
\end{equation*}
where $B(I_0)$ is the family of measurable maps from $I_0$
to $I_0$ with the uniform norm:
\begin{align*}
  \|F(\tilde\theta)-F(\theta)\|=\sup_{x\in
I_0}|f_{\tilde\theta}(x)-f_\theta(x)|.
\end{align*}
{We write $\Disc_{F}$ for the set of discontinuities of the
map $F$}.  We further assume some regularity of the map $F$.
\begin{enumerate}
\item[$(H_3)$] $\nu(\Disc_{F})=0$. 
\end{enumerate}

We deal with two situations:
\begin{enumerate}
\item[$(H_4)$] the maps $f_\te$ are $C^3$, $Sf_\te\leq 0$, for
  every $\te\in\toro$ (here $Sf_\theta$ is the Schwarzian
  derivative of
  $f_\theta$) 
  and the derivatives of $\{f_{\te}\}_{\te\in\toro}$ are
  equicontinuous.\footnote{The equicontinuity can be replaced
    by the following condition: given $\epsilon>0$, there
    exists $\delta>0$ such that if $|x-\crit_\te|<\delta$
    then $|f'_{\te}(x)|<\epsilon$, for all
    $\te\in\toro$. This is used in the proof of
Theorem~\ref{PrincipalB}.}

\item[$(H^*_4)$] we have $\disc_\te\neq\emptyset$ for some
  $\theta\in\toro$. Writing $\disc_\te=\{q_1(\te)\leq\ldots
  \leq q_{d(\theta)}(\te)\}$ (this may be the empty set for
  some values of $\theta\in\toro$) for every $\te\in\toro$,
  we assume that $f_\theta$ is $C^3$ diffeomorphism and
  ${Sf_\te}\leq 0$ restricted to $(q_i(\te), q_{i+1}(\te))$ for
  all $i= 0,1,\ldots,d(\theta)$, where we set $q_0=\inf I_0$
  and $q_{d(\theta)+1}=\sup I_0$ to be the endpoints of
  $I_0$.

  Writing $\disc=\{(\theta,x): x\in\disc_\theta,
  \theta\in\toro\}$ we also assume that for every
  $\ell\in\Z^+$ there exists a neighborhood $V$ of $\overline{\disc}$
  such that
  \begin{align*}
    \vfi^{k}(V)\cap V=\emptyset
    \quad\text{for every}\quad k=1,\dots,\ell.
  \end{align*}
\end{enumerate}
We write, here and in the rest of the paper, $\overline C$
for the topological closure of a subset $C\subset
\toro\times I_0$.
 
This setting models similar maps as in
\cite{gouezel07,solano}, but without expansion assumptions
on the base, and we also admit discontinuities but with
strong non-recurrence assumptions. This non-recurrence
property can be deduced, as in
Example~\ref{ex:intermittent}, if every sequence $z_k$ in
$\toro\times I_0$ tending to $\disc$ is sent to a
sequence $\vfi(z_k)$ tending to a forward invariant subset
disjoint from $\overline{\disc}$; a sort of
Misiurewicz condition, but this time on the images of a
discontinuity set.

We say that $\fhi$ is \emph{non-uniformly expanding
  along the vertical direction according to
  $\nu\times\leb$}, if
\begin{equation}
  \label{eq:posiexponents}
\ilim \frac{1}{n} \log |Df_\te^n(x)|>2\lambda \quad
(\nu\times\leb)-\text{a.e. } (\te,x)
\end{equation}
for some $\lambda>0$, where $\leb$ denotes the normalized Lebesgue measure on
$I_0$ and we use the convention
\begin{equation*}
f_\te^k(x):=f_{\al^{k-1}(\te)}\circ \ldots f_{\al^{1}(\te)}\circ
f_{\te}(x)
\end{equation*}	
for every $\te\in \toro$, $x\in I_0$.  We say that $\fhi$ is \emph{non-uniformly expanding along the vertical
  direction according to $\nu\times\leb$, on the subset
$Z\subset\toro\times I_0$}, if \eqref{eq:posiexponents} holds for
$\nu\times\leb$-a.e. $(\te,x)\in Z$, for some $\lambda>0$.

We recall that for an ergodic $\vfi$-invariant probability
measure, its \emph{ergodic basin} is the set
\begin{align*}
  B(\mu)=\left\{\omega=(\theta,x)\in\toro\times\Y:
  \lim_{n\to+\infty}\frac1n\sum_{j=0}^{n-1}g(\fhi^j(\omega))=\int
  g\,d\mu
  \quad\text{for each}\quad g\in C^0(\toro\times\Y,\R)\right\}.
\end{align*}
Our main result in this setting is the following

\begin{maintheorem}
  \label{mcor:finiteergodicmeasures}
  Let $\fhi: \toro\times I_0\ra \toro\times I_0$ be a
  skew-product as above satisfying $(H_1)$, $(H_2)$, $(H_3)$
  and $(H_4)$ (or $(H^{*}_4)$).  Assume that $\fhi$ is
  non-uniformly expanding along the vertical direction
  according to $\nu\times\leb$, on the subset $Z\subset
  \toro\times I_0$. Then $\fhi$ admits finitely many ergodic
  invariant probability measures absolutely continuous with
  respect to $\nu\times \leb$, whose basins cover $Z$, up to
  a $\nu\times\leb$-zero measure set.
\end{maintheorem}  

Note that the existence of an invariant measure for the base
dynamics (see condition $(H_2)$) is not a restriction in the
theorem. Indeed, any $\fhi$-invariant measure absolutely
continuous (with respect to $\mu_{\toro}\times\leb$, where
$\mu_{\toro}$ is a measure on $\mathcal{B}_\toro$) induces
an $\al$-invariant measure which is absolutely continuous
(with respect to $\mu_{\toro}$).

In the case that the rate of expansion is not bounded away
from zero, we have a weaker result.

\begin{maintheorem}
  \label{mthm:PrincipalA}
  Let $\fhi: \toro\times I_0\ra \toro\times I_0$ be a
  skew-product as above satisfying $(H_1)$, $(H_2)$,
  $(H_3)$ and $(H_4)$ (or $(H^{*}_4)$). Assume that
  the limit in \eqref{eq:posiexponents} is greater than 0, for
  $\nu\times\leb$-a.e. $(\te,x)\in Z$.
Then $\fhi$ admits an at most denumerable family
  $\{\mu_i\}_{i\in I}$ of \emph{ergodic} invariant
  probability measures absolutely continuous with respect to
  $\nu\times \leb$. Moreover $\nu\times\leb$-a.e.
  $(\theta,x)\in Z$ belongs to the basin of some $\mu_i,i\ge1$.
\end{maintheorem}


\subsubsection{Random dynamical systems interpretation}
\label{sec:random-dynamic-resul}
Let $(\toro, \mathcal{B}_{\toro},\nu)$ be a probability
space and let $\al$ be an $\nu$-preserving measurable map on
$\toro$.  A random dynamical system $f$ on the measurable
space $(\Y, \mathcal{B}_{\Y})$ over $(\toro,
\mathcal{B}_{\toro},\nu, \al)$ is generated by mappings
$f_\te$, $\te\in\toro$, so that (see \cite[Definition
1.1.1]{arnold-l-1998}):
\begin{enumerate}
\item the map $(\te,x)\to f_\te(x)$ is measurable, and
\item it satisfies the cocycle property
  $f_{\te}^{n+m}=f_{\al^{m}(\te)}^{n}\circ f_{\te}^{m}$ for
  all $ n,m\in\Z^+,\theta\in\toro$.
\end{enumerate}
The associated
random orbits are $x_0, x_1, \ldots$, where $x_0\in \Y$ and
$x_{n+1}=f_{\al^{n}(\te)}(x_n)$. This random dynamical
system (RDS for short) is denoted by $(\toro,
\mathcal{B}_{\toro},\nu, \al,f)$.

In general there is no common measure invariant for all the maps
$f_\te$, $\te\in \toro$. But one can ask whether there exists a
measure (or a finite number of measures) describing the asymptotics of
almost all random orbits, in the sense defined to follow.
Let us denote by $\delta_{x}$ the Dirac measure at $x$.

\begin{definition}
A probability measure $\mu$ on $\Y$ is
  \emph{SRB} for the RDS $(\toro,
  \mathcal{B}_{\toro},\nu, \al,f)$ if, for
  $\nu$-~almost every $\te\in\toro$, the set
  $RB_\te(\mu)$ of points $x\in \Y$ such that
\begin{equation*}
 \frac{1}{n}\sum_{k=0}^{n-1} \delta_{f_{\al^{k-1}(\te)}\circ\dots\circ
f_\te(x)}\longrightarrow \mu
\end{equation*}
has positive Lebesgue measure.
We call
$RB_\te(\mu)$ the random basin of $\mu$.
\end{definition}
One can associate to the random dynamical system $f$ the skew product
$\fhi:\toro\times \Y\circlearrowleft$, $(\te,x)\mapsto
(\al(\te), f_\te(x))$.
Note that, a $\fhi$-invariant measure $\mu$ with marginal
$\nu$, that is, such that $\mu(A\times I_0)=\nu(A)$ for
every $\nu$-measurable $A\subset\toro$, is an
\emph{invariant measure for the random dynamical system
  $(\toro, \mathcal{B}_{\toro},\nu, \al,f)$}; see
\cite[Definition 1.4.1]{arnold-l-1998}. All the
$\fhi$-invariant measures obtained in
Theorems~\ref{mcor:finiteergodicmeasures}, \ref{mthm:PrincipalA} and
\ref{mthm:PrincipalB} are of this type; see
Lemma~\ref{le:marginal-nu} in
Section~\ref{subs:basic-invariant-measures}.

We say that the random map 
$(\toro, \mathcal{B}_{\toro},\nu, \al,f)$ is a
\begin{itemize}
 \item \emph{random non-uniformly expanding map on $I_0$}
if $\toro$ is a Polish space, $\Y=I_0$ and the 
associated skew-product is non-uniformly expanding along the 
vertical direction according to $\nu\times\leb$.
\item \emph{admissible random non-uniformly expanding map on $I_0$}
if it is a random non-uniformly expanding map on $I_0$
and the associated 
skew-product satisfies $(H_1)$, $(H_2)$, $(H_3)$ and 
$(H_4)$ (or $(H^{*}_4)$).
\end{itemize}

We can state similar results to Theorems
\ref{mcor:finiteergodicmeasures} and \ref{mthm:PrincipalA} 
in the setting of
random non-uniformly expanding maps, since the
associated skew-product satisfies the conditions of
these results. Moreover, inspired by one result of
Buzzi \cite[Theorem 0.5]{Buzz00}, we can state the
following probabilistic consequence of our results.

\begin{maintheorem} \label{mthm:finitelymanySRB} Any
  admissible random non-uniformly expanding map on $I_0$
  admits a finite number of SRB measures. Moreover, the SRB
  measures are absolutely continuous and, $\nu$-almost
  surely, the union of their random basins has total
  Lebesgue measure.
\end{maintheorem}

We observe that if $\toro=\Sigma^\N$, where $\Sigma$ is an
at most countable set, then $\toro$ is totally
disconnected. In addition, setting
$f_\theta=f_{\pi(\theta)}$ where $\pi:\toro\to\Sigma^k$ is a
projection on the first $k$-symbols of $\theta\in\toro$, and
$\alpha$ the left shift of $\Sigma^\N$ we have both
$\Disc_{\alpha}=\emptyset$ and $\Disc_F=\emptyset$, since
$f_\theta$ depends only on finitely many coordinates of the
point $\theta\in\toro$ (the map $F:\toro\to B(I_0)$ is
locally constant).

Hence we obtain the following as a corollary of Theorem
\ref{mthm:finitelymanySRB}.
\begin{maincorollary} \label{cor:finitelySRBquadratic}
  Let $f_i:I_0\to I_0, i\in\Sigma$ be a countable family of
  maps of the quadratic family, that is,
  $f_i(x)=f_{\theta_i}(x)=\theta_i-x^2$ with
  $\theta_i\in[1,2]$. Let also $\toro=\Sigma^\N$ and
  $\alpha:\toro\circlearrowleft$ be the left shift with some
  ergodic $\alpha$-invariant probability measure $\nu$. 

  If $(\Sigma^\N, \mathcal{B}_{\Sigma^\N},\nu, \al,f)$ is a random 
non-uniformly expanding map on $I_0$, then it admits a finite 
number of SRB measures. The SRB 
measures are abolutely continuous and the union of their
random basins has total Lebesgue measure $\nu$-a.e.
\end{maincorollary}
Similar results holds for families of maps satisfying the
non-uniformly expanding conditions with higher-dimensional
fibers, as we state in the following
Section~\ref{sec:higher-dimens-fibers}.

\subsubsection{Higher-dimensional fibers}
\label{sec:higher-dimens-fibers}

Assuming a condition of slow recurrence to the set of
criticalities and/or discontinuities, which we assume are of
a certain non-degenerate type, we can take advantage of the
method of proof of Theorems~\ref{mcor:finiteergodicmeasures},~\ref{mthm:PrincipalA}
and~\ref{mthm:finitelymanySRB} to obtain the same conclusion
in a setting where the fibers can be higher dimensional
manifolds.

Let us assume that $\vfi:\toro\times\Y\to\toro\times\Y$ has
the same skew-product form as before, but now:
\begin{enumerate}
\item[$(H_5)$] $f:\toro\times\Y\to\Y$ is a Borel measurable
  map such that $f_\theta:\{\theta\}\times\Y\to\Y$ is
  $C^{1+\alpha}$ away from a set of non-degenerate
  discontinuities $\disc_\theta$ and/or criticalities
  $\crit_\theta$ in the compact finite $d$-dimensional
  manifold $\Y$.
\end{enumerate}
We fix a Riemannian metric on
$\Y$, the corresponding distance function $\dist$ and norm
$\|\cdot\|$ to be used in what follows. We also fix a
normalized volume form $\Leb$ (Lebesgue measure) on
$\Y$. The next regularity conditions on the derivatives will
be needed.
\begin{enumerate}
\item[$(H_6)$] $f':\toro\times\Y\to\cL(\R^d,\R^d),
  (\theta,x)\mapsto Df_\theta(x)$ and
  $f'_1:\toro\times\Y\to\cL(\R^d,\R^d), (\theta,x)\mapsto
  Df_\theta(x)^{-1}$ are Borel measurable maps with respect
  to the Borel $\sigma$-algebras of $\toro\times\Y$ and
  $\cL(\R^d,\R^d)$. In this last space we consider
  the topology induced by the usual operator
  norm $\|L\|_{\theta,x}:=\sup\{\|L(v)\|_{f_\theta(x)}/\|v\|_x:
  \vec0\neq v\in T_x\Y\}$ for a linear map $L:T_x\Y\to
  T_{f_\theta(x)}\Y, (\theta,x)\in\toro\times\Y$.
\end{enumerate}
We also assume conditions $(H_1)$ and $(H_2)$ (or $(H_2^*)$)
and $(H_3)$ on $\cc$, $\Disc_\alpha$ and $\Disc_F$ as
before replacing $I_0$ by $\Y$ throughout. 

The non-degenerate assumption on the sets {$\crit_\theta$ and
$\disc_\theta$} mean that $f_\theta$ behaves like a power of
the distance near the set of
criticalities/discontinuities. More precisely: there are
constants $B>1$ and $\be>0$ for which, {writing
$\cc_\theta$ for $\cc\cap(\{\theta\}\times\Y)$}
\begin{itemize}
 \item[(S1)]
\hspace{.1cm}$\displaystyle{\frac{1}{B}
\dist(x,\cc_\theta)^{\be}
\leq
\frac{\|Df_\theta(x)v\|}{\|v\|}
\leq 
B\dist(x,\cc_\theta)^{-\be}}$;
 \item[(S2)]
\hspace{.1cm}$\displaystyle{\left|\log\|Df_\theta(x)^{-1}\|-
\log\|Df_\theta(y)^{-1}\|\:\right|\leq
B\frac{\dist(x,y)}{\dist(x,\cc_\theta)^{\be}}}$;
 \item[(S3)]
\hspace{.1cm}$\displaystyle{\left|\log|\det Df_\theta(x)^{-1}|-
\log|\det Df_\theta(y)^{-1}|\:\right|\leq
B\frac{\dist(x,y)}{\dist(x,\cc_\theta)^{\be}}}$;
 \end{itemize}
 for every $\theta\in\toro$ and $x,y\in
 \Y\setminus(\cc_\theta)$ with
 $\dist(x,y)<\dist(x,\cc_\theta)/2$ and $v\in T_x\Y$.

Given $\delta>0$ we define the $\delta$-{\em truncated
  distance} from $x\in \Y$ to $\cc_\theta$
 $$ \dist_\delta(x,\cc_\theta)= \left\{
\begin{array}{ll} 
1 & \mbox{if }\dist(x,\cc_\theta)\geq \delta,\\
\dist (x,\cc_\theta) & \mbox{otherwise.} \end{array} \right. $$

We say that $\vfi$ is \emph{non-uniformly expanding along
  the fibers according to $\nu\times\Leb$, on
  $Z\subset\toro\times\Y$,} if
\begin{itemize}
\item $\vfi$ has \emph{non-uniform expansion along the
    vertical direction according to $\nu\times\Leb$ on $Z$}:
  for some $\lambda>0$,
\begin{align}\label{eq:pos-lyap-vert}
    \limsup_{n\to+\infty}\frac1n\sum_{j=0}^{n-1}
    \log\|Df_{\alpha^j(\theta)}(f_\theta^j(x))^{-1}\|<- 2\lambda,\quad
    \nu\times\Leb-\text{a.e  } (\theta,x)\in Z;
  \end{align}
\item $\vfi$ has \emph{slow recurrence to the set of
    criticalities and discontinuities} on the orbit of
  points of $Z$: for each
  $\epsilon>0$ there exists $\delta>0$ such that
  \begin{align}\label{eq:slow-rec}
    \limsup_{n\to+\infty}\frac1n\sum_{j=0}^{n-1} -\log
    \dist_\delta\big(f^j_{\theta}(x),\cc_{\alpha^j(\theta)}\big)
    <\epsilon, \quad
    \nu\times\Leb-\text{a.e  } (\theta,x)\in Z
  \end{align}
  (the reader can recall the definition of $\cc$ in the
  statement of condition $(H_1)$).
\end{itemize}

Our result in this setting reads as follows.

\begin{maintheorem}
  \label{mthm:PrincipalB}
  Let $\fhi: \toro\times \Y\ra \toro\times \Y$ be a
  skew-product as above satisfying $(H_1)$, $(H_2)$ $(H_3)$,
  $(H_5)$ and $(H_6)$.  Assume that $\fhi$ non-uniformly
  expanding along the fibers according to $\nu\times\Leb$, on the subset
  $Z\subset\toro\times\Y$.  Then we obtain the same
  conclusions as in Theorem~\ref{mcor:finiteergodicmeasures}.
\end{maintheorem}

In this setting, we also have an analogue of Theorem
\ref{mthm:PrincipalA}: if the limit in
(\ref{eq:pos-lyap-vert}) is smaller than zero, then $\fhi$
admits an at most denumerable family $\{\mu_i\}_{i\in I}$ of
ergodic invariant probability measures absolutely continuous
with respect to $\nu\times \Leb$, whose basins cover
$Z$. The proof is identical to the deduction of the
statement of Theorem \ref{mthm:PrincipalA} from that of
Theorem~\ref{mcor:finiteergodicmeasures}.


\subsection{Strategy of the proof and
organization of the text}
\label{sec:organiz-text}

The basic idea is to define measures on the vertical
foliation of the skew-product, depending on the starting
vertical leaf $\{\theta\}\times I_0$ or
$\{\theta\}\times\Y$; show that these measures depend
measurably on $\theta\in\toro$ and can be integrated with
respect to $\nu$; and then show that weak$^*$ accumulation
points of these integrated measures are $\vfi$-invariant.

The assumption of non-uniform expansion along the
vertical direction, or along the fibers, enables us to control the
densities of these measures along the vertical direction on
a certain subset of points which has ``positive mass at
infinity''. This provides us with an absolutely continuous
component for every weak$^*$ accumulation point obtained
before. Finally, using the uniqueness of Lebesgue
decomposition and the smoothness assumption on $f_\theta$
allows us to obtain an invariant probability measure $\mu$
for the skew-product $\vfi$ which is absolutely continuous
with respect to the product measure $\nu\times\leb$ of the
invariant measure on the base and Lebesgue measure on the
interval. In the case of Theorem \ref{mthm:PrincipalB},
the absolute continuity is respect to $\nu\times\Leb$, where $\Leb$ 
is the Lebesgue measure on $\Y$.
The ergodicity is obtained as a consequence of
the fact that the invariant sets, with positive
$\nu\times\leb$-measure, have $\nu\times\leb$-measure
bounded away from zero.

In the next Section~\ref{sec:exampl} we present some
examples of application our main results. In
Section~\ref{subs:basic-invariant-measures} we
construct the basic measures we will use to obtain the
invariant probability measures for $\vfi$. In
Section~\ref{sec:acim} we construct an absolutely
continuous invariant probability measure for $\vfi$. In
Section~\ref{sec:finitely-many-ergodi}, we prove that
the invariant sets with positive measure have
measure uniformly bounded away from zero. As consequence of this
result, we conclude the existence of ergodic absolutely
continuous invariant probabilities. From these
arguments it also follows the conclusion of Theorems
\ref{mcor:finiteergodicmeasures} and
\ref{mthm:PrincipalA}. In Section
\ref{sec:SRBrandommap} we prove Theorem
\ref{mthm:finitelymanySRB}, about existence of finitely
many SRB probability measures for random non-uniformly
expanding maps.

In Sections \ref{subs:basic-invariant-measures} and
\ref{sec:acim} we assume that the base dynamics
$\alpha:\toro\circlearrowleft$ is a bimeasurable bijection.
We explain how to replace this condition by $(H_2)$ in
Section~\ref{sec:Remove-H2}. Finally, in
Section~\ref{sec:higher-dimens-fibers-0} we outline the
arguments proving the main theorems in the setting with
higher-dimensional fibers; and in Appendix~\ref{sec:measur}
we prove the measurability of the sets used in the
construction of the measures in the previous sections.

\subsection*{Acknowledgments}
The authors thank Universidade Federal do Rio de
Janeiro (UFRJ) and IMPA, at Rio de Janeiro, Brasil, and
also Pontificia Universidade Catolica de Valparaiso
(PUCV), at Valparaiso, Chile, where part of this work
was developed, for their hospitality.

We thank the anonymous referee for the detailed suggestions
that helped improved the presentation and the readability of
the text.

\section{Some examples and open problems}
\label{sec:exampl}

As mentioned in Section~\ref{sec:random-dynamic-resul},
every skew-product map
$\vfi(\theta,x)=(\alpha(\theta),f_\theta(x))$ on
$\toro\times\Y$ presented below can be seen as a RDS
$(\toro, \mathcal{B}_{\toro},\nu, \al,f)$ in a standard way;
see \cite[Definition 1.1.1]{arnold-l-1998}.

\begin{example}\label{ex:vianamaps}
  Skew-products of quadratic maps have been extensively studied. In
\cite{Vi97,buzzi-sester-tsujii} is proved \eqref{eq:posiexponents},
with $\nu$ being Lebesgue measure on $\circulo^1$, for the maps
\begin{align*}
    F: \circulo^1\times \R \to \circulo^1\times\R,
    (\theta,x)\mapsto (k\cdot\theta,a_0-x^2+ a\sin(2\pi\theta))
  \end{align*}
where $k\in\Z^+\setminus\{1\}$ and $a_0\in(1,2]$ is such that $0$ is
preperiodic for the map $f_{a_0}(x)=a_0-x^2$. In \cite{schnell08} the
same map $F$ as above was studied
  but with $k$ a real parameter in the interval
  $(R_0,+\infty)$, where $1<R_0<2$ was shown to exist so
  that, the map $F$ with $k>R_0$
  satisfies~\eqref{eq:posiexponents}.
 
 In~\cite{schnell09} were considered skew-products
$ G(\theta,x)=(f_{a_1}^k(\theta),f_{a_0}(x)+\alpha s(\theta))$,
 where $f_a(x):=a-x^2$ and $a_0,a_1$ are parameters in the
  interval $(1,2]$ such that the critical point is
  pre-periodic but not periodic, and $s:\circulo^1\to[-1,1]$
  is a piecewise $C^1$ map.  It was proved that there exist
  $k_0\in\Z^+$ and a $C^1$ map $s$ such
  that, for every small enough $\alpha>0$ and all integers $k\ge
  k_0$, the map $G$ satisfies~\eqref{eq:posiexponents}, with
  $\toro=[f_{a_1}^2(0),f_{a_1}(0)]$ and $\nu$ being Lebesgue
  measure on the invariant interval $\toro$. 
  
  Note that the base transformation for the maps in
\cite{Vi97,buzzi-sester-tsujii,schnell08} is (piecewise) expanding.
For the maps in \cite{schnell09}, it is non-uniformly expanding with
critical points.
 
The existence of absolutely continuous invariant
probability measures for all these maps is an immediate
consequence of Theorem~\ref{mcor:finiteergodicmeasures}, with
$\toro=\circulo^1$ and $\vfi=F$ or $\vfi=G$.

Let us mention that the construction of the absolutely
continuous invariant probability was obtained
in~\cite{Al00} for the maps considered on
\cite{Vi97,buzzi-sester-tsujii}. In \cite{schnell08}
this conclusion was only achieved for a full Lebesgue
measure subset of $(R_0,+\infty)$. The author
in~\cite{schnell09} did not obtain absolutely
continuous invariant measures. Recently, in
\cite{AlSchn} was obtained the result for all the maps
in \cite{schnell08,schnell09}, as a byproduct of the
application of inducing to study decay of correlations
for the unique absolutely continuous invariant
probability measure.

\end{example}
\begin{example}
  \label{ex:intermittent}
  We can produce examples where the base dynamics is
  essentially arbitrary. Let $\toro$ be the circle
  $\circulo^1$ and $\alpha:\circulo^1\to\circulo^1$ a measurable map
  preserving an ergodic probability measure $\nu$. Let
  $\theta\mapsto f_\theta$ be a continuous family of maps of
  the interval $I_0=[0,1]$ such that
  \begin{itemize}
  \item for all $\theta\in\circulo^1$ the map $f_\theta:I_0\to
    I_0$ is $2$-to-$1$, with two branches
    $f_\theta\mid[0,1/2]: [0,1/2]\to[0,1]$ and
    $f_\theta\mid[1/2,1]: [1/2,1]\to[0,1]$ both increasing
    diffeomorphisms;
  \item on an arc $A$ of $\circulo^1$ with $\nu(A)\ge 1-
    \epsilon$ for some small $\epsilon>0$ we have
    \begin{itemize}
    \item for $\theta\in A$ the map $f_\theta$ is expanding:
      there exists $\sigma>1$ such that
      $|Df_\theta(x)|\ge\sigma$ for all $x\in I_0$;
    \item for $\theta\in\circulo^1\setminus A$ the map
      $f_\theta$ does not contract too much: there exists
      $\delta>0$ small such that $|Df_\theta(x)|\ge
      1-\delta$ for all $x\in I_0$.
    \end{itemize}
  \end{itemize}
  In this setting we have that for
  $(\nu\times\leb)$-a.e. $(\theta,x)$, applying the
  Ergodic Theorem to the sequence $(\alpha^j(\theta))_{j\ge0}$
  \begin{align*}
    \liminf_{n\to+\infty}\frac1n\sum_{j=0}^{n-1}
    \log|Df_{\alpha^j(\theta)}(f_\theta^j(x))|
    &\ge
    \nu(A)\log\sigma + \nu(\circulo^1\setminus
    A)\log(1-\delta)
    \\
    &\ge
    (1-\epsilon)\log\sigma -\delta\epsilon >0,
  \end{align*}
  where $\leb$ is the Lebesgue measure on $I_0$.  

  For a concrete expression we may take (see Figure~\ref{fig1})
  \begin{equation}
  \label{eq:intermittentfamily}
  f_t(x)=\left\{
\begin{array}{ll}
tx + 2^{\beta}(2-t) x^{1+\beta} \qquad &\text{if  }  x \in
[0,\frac{1}{2})\\
1-t(1-x) - 2^{\beta}(2-t) (1-x)^{1 + \beta}
&  \text{if  } x \in [\frac{1}{2},1]
\end{array} \right.
\end{equation}
with $\beta\in(0,1)$ and $t\in(1/2,3/2)$. We then take a
function $t:\circulo^1\to(1/2,3/2)$ such that, for some
small $a>0$, satisfies $t(A)\subset (1+a,3/2)$ and
$t(\circulo^1\setminus A)\subset (1-a,1+a]$. Finally we
define $\vfi(\theta,x)=(\alpha(\theta),f_{t(\theta)}(x))$.

We remark that $\disc=\circulo^1\times \{1/2\}$ is such that
every sequence $z_k$ converging to $\disc$ on
$\circulo^1\times I_0$ is sent to a sequence $\vfi(z_k)$
whose accumulation points are contained in
$\circulo^1\times\{0,1\}$, which is a forward invariant
subset of $\vfi$. This implies the strong non-recurrence
condition in $(H_4^*)$.

\begin{figure}[htbp]
 \centering
 \includegraphics[width=12cm, height=6cm]{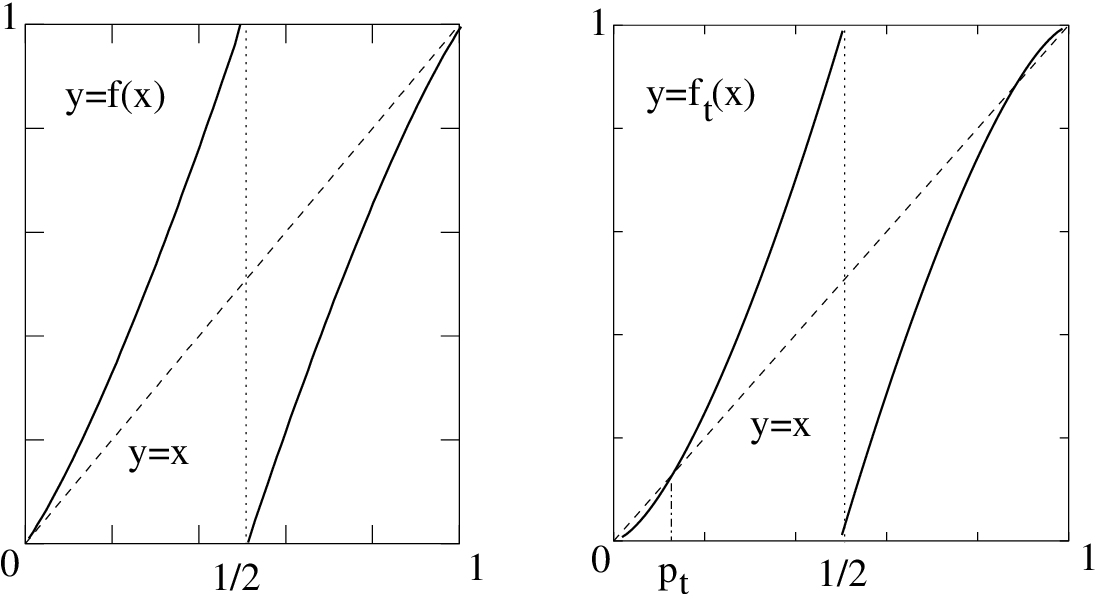}
 \caption{The map $f_1$ (left) and the map $f_t$ for $t<1$ (right).}
 \label{fig1}
\end{figure}

From Theorem~\ref{mthm:PrincipalA} we have that $\vfi$
admits an invariant probability measure $\mu$
absolutely continuous with respect to $\nu\times\leb$.
\end{example}

%
%

\begin{remark}\label{rem:non-atomic-singular-base}
We can construct this example with $\alpha$ a circle
diffeomorphism with irrational rotation number and $\nu$ an
ergodic $\alpha$-invariant probability which is non-atomic
and singular with respect to $\leb$; see e.g.
\cite[Theorem 12.5.1]{KH95}. We note that in this way we
have a \emph{base map $\alpha$ with no average expansion}.
\end{remark}

%

\begin{example}
  \label{ex:higher-dim}
  We can adapt the construction in
  Example~\ref{ex:intermittent} with fibers of arbitrary
  dimension. We fix $k>1$ in what follows.

  Let again $\toro$ be the circle
  $\circulo^1$ and $\alpha:\circulo^1\to\circulo^1$ a measurable map
  preserving an ergodic probability measure $\nu$. Let now
  $\theta\mapsto f_\theta$ be a continuous family of maps of
  the $k$-torus $\T^k$ such that, as before,
  \begin{itemize}
  \item on an arc $A$ of $\circulo^1$ with
    $\nu(A)\ge1-\epsilon$ for some small $\epsilon>0$ and
    some Riemannian norm $\|\cdot\|$ on $\T^k$ we have:
    \begin{itemize}
    \item for $\theta\in A$ the map $f_\theta$ is expanding:
      there exists $\sigma>1$ such that
      $\|Df_\theta(x)^{-1}\|\le1/\sigma$ for all $x\in\T^k$;
    \item for $\theta\in\circulo^1\setminus A$ the map
      $f_\theta$ does not contract too much: there exists
      $\delta>0$ small such that
      $\|Df_\theta(x)^{-1}\|\le1+\delta$ for all $x\in\T^k$.
    \end{itemize}
  \end{itemize}
  As before, in this setting, we have for
  $(\nu\times\Leb)$-a.e. $(\theta,x)$ that, applying the
  Ergodic Theorem to the sequence $(\alpha^j(\theta))_{j\ge0}$
  \begin{align*}
    \limsup_{n\to+\infty}\frac1n\sum_{j=0}^{n-1}
    \log\|Df_{\alpha^j(\theta)}(f_\theta^j(x))^{-1}\|
    &\le
    \nu(A)\log\sigma + \nu(\circulo^1\setminus
    A)\log(1+\delta)
    \\
    &\le
    (1-\epsilon)\log\sigma +\delta\epsilon <0,
  \end{align*}
  where $\Leb$ is the some volume from (Lebesgue measure)
  on $\T^k$. Since there are no criticalities or
  discontinuities, this shows that
  $\vfi(\theta,x)=(\alpha(\theta),f(\theta,x))$ is a
  non-uniformly expanding map along the fibers and we may
  apply Theorem~\ref{mthm:PrincipalB} to conclude the
  existence of a probability measure $\mu$ absolutely
  continuous with respect to $\nu\times\Leb$.
\end{example}

\begin{example}
  \label{ex:discontinuous}
  Now we adapt the previous Example~\ref{ex:higher-dim} to
  have a discontinuous family of fiber maps. We repeat the
  construction, keeping the choice of $f_\theta$ for
  $\theta\in A$ but replacing $f_\theta$ by the identity map
  on the torus for $\theta\in\circulo^1\setminus A$.

  We still have non-uniform expansion and we note that the
  discontinuities of the map $F$ are on the boundary
  $\partial A$ of the arc $A$ of the circle, which is formed
  by a two points on the circle. Hence condition $(H_3)$ is
  satisfied. We apply Theorem~\ref{mthm:PrincipalB} to
  obtain a $\vfi$-invariant probability $\eta$ absolutely
  continuous with respect to $\nu\times\Leb$.
\end{example}

\begin{example}
  \label{ex:infinite-srb}
  We present an example of a $C^\infty$ map $T$ away from a
  denumerable singular set, which is non-uniformly
  expanding and has infinitely many ergodic absolutely
  continuous invariant probability measures. 

  On the one hand, considering $\alpha=T$ as the base map
  and a constant fiber map $f(x)=4x(1-x)$ of the interval
  which has positive Lyapunov exponents for Lebesgue almost
  all point, a unique critical point and negative Schwarzian
  derivative, we obtain a direct product $\vfi=\alpha\times
  f$. The map $f$ admits a unique ergodic absolutely continuous
  invariant probability measure $\mu$.  Thus we can apply
  our arguments to each ergodic absolutely continuous
  invariant probability measure $\nu_k$ for $\alpha$ to
  obtain $\nu_k\times\mu$ as an ergodic absolutely
  continuous invariant probability measure for $\vfi$. In
  this way $\vfi$ has a countable set of distinct absolutely
  continuous invariant probability measures.

  On the other hand, considering the direct product
  $\vfi=\alpha\times T$ of any map $\alpha$ of a metric
  space with an ergodic probability measure $\nu$, with $T$
  on the fibers, we obtain an example with infinitely many
  ergodic absolutely continuous invariant measures
  $\nu\times\nu_k$ with the same marginal $\nu$.

  The map $T$ is easily described as the standard doubling
  map
  \begin{align*}
    f: x \in[0,1] \mapsto
    \begin{cases}
      2x & \text{if  } 0\le x <1/2
      \\
      2x-1 & \text{if  } 1/2\le x \le 1
    \end{cases}
  \end{align*}
  conveniently rescaled on the unit interval infinitely many
  times, as follows, see figure~\ref{fig2}:
  \begin{align*}
    T(x):=\sum_{n\ge1}
    \begin{cases}
        \frac1{2^n}+
          \frac1{2^n}f\big(2^n(x-2^{-n})\big)
        & \text{if  }
        x\in\Big]\frac1{2^n},\frac1{2^{n-1}}\Big]
        \\
        0 & \text{otherwise}
      \end{cases}.
  \end{align*}

\begin{figure}[htbp]
\centering
\includegraphics[width=5cm, height=5cm]{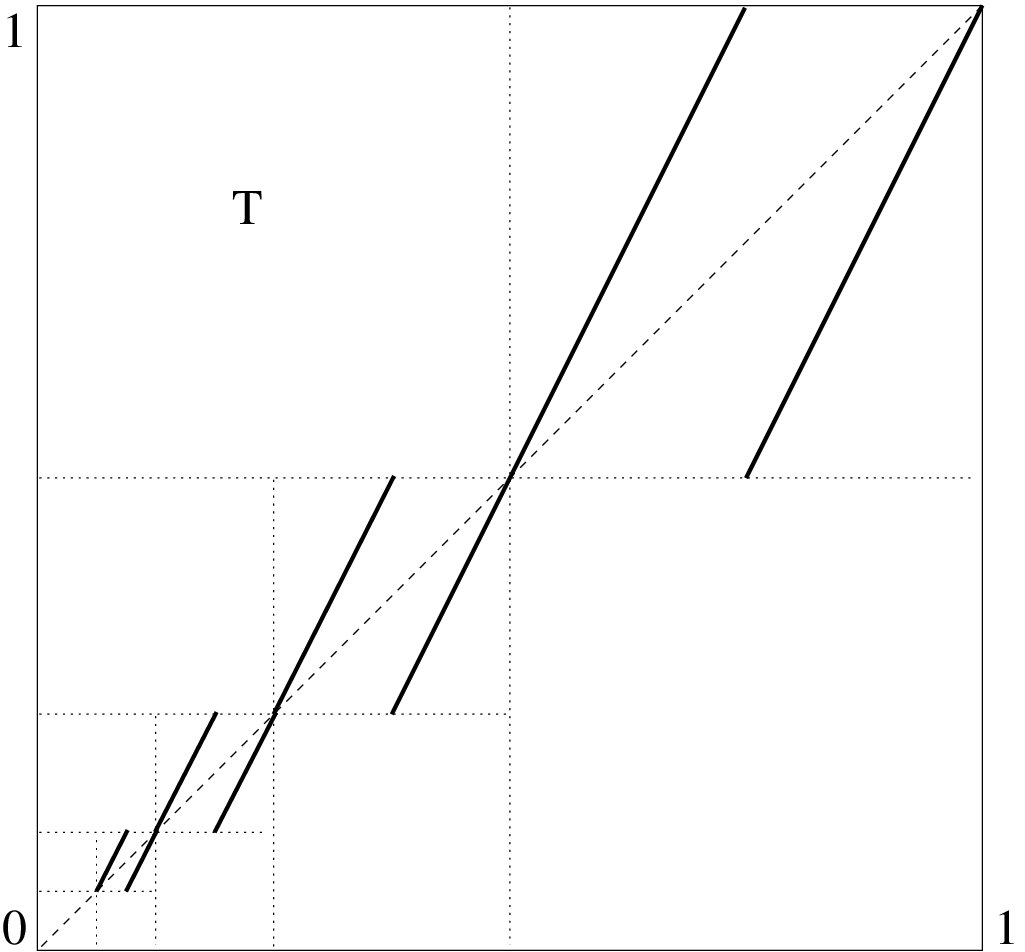}
\caption{A sketch of the map $T$.}
\label{fig2}
\end{figure}
It is clear that $DT\equiv2$ and $DT^2\equiv0$ outside the
compact set $\cc:=\{0\}\cup\{2^{-n}, 2^{-n}+2^{-(n+1)}:
n\in\Z^+\}$. It is easy to see that Lebesgue measure
$\leb$ on $[0,1]$ is invariant and each interval
$[2^{-n},2^{-n+1}]$ supports an ergodic component of
$\leb$ given by the normalized restriction of $\leb$
to this interval.


  Moreover it is straightforward to check that the set $\cc$
  satisfies conditions $(S1)$ through $(S3)$ with constants
  $B=\beta=1$, so $\cc$ is a non-degenerate singular set for
  $T$. In addition, conditions $(H_2), (H_3)$ and $(H_4^*)$
  are also easily checked. 

  However the slow recurrence condition is not satisfied:
  for each given $\delta>0$ and $N>1$ there exists $k>N$
  such that $2^{-k+1}<\delta$ and we have
  \begin{align*}
\liminf_{n\to+\infty}\frac1n\sum_{j=0}^{n-1}-\log\dist_\delta(T^j(x),
\cc)
    \ge k > N \quad\text{for all}\quad
  x\in(2^{-k},2^{-k+1}). 
\end{align*}
But this condition fails in a small set: for each $N>1$ the
points for which the above inequality holds are contained in
$[0,2^{-[\log_2 N]+1})$, where $[x]$ denotes the integer part
of $x$.
\end{example}

\subsection{Problems}
\label{sec:problems}

We list below some open problems related with our setting of
random non-uniformly expanding maps.

\begin{enumerate}
 \item \label{prob:vianamaps}
 Consider the family $f_\theta(x)=a_0+\theta-x^2$ of
  quadratic maps of $\Y=\R$ as in Example \ref{ex:vianamaps}, set
  $\toro=[-\epsilon,\epsilon]^\N$ for some fixed
  $\epsilon>0$ and let $\alpha:\toro\to\toro$ be the left
  shift map on $\toro$ endowed with the ergodic invariant
  measure $\nu=\lambda_\epsilon^\N$, where
  $\lambda_\epsilon$ is Lebesgue measure on $[-\epsilon,\epsilon]$. 
\noindent
 Is $\vfi(\theta,x)=(\alpha(\theta),f_{\theta_1}(x))$
  non-uniformly expanding for random orbits for some
  parameter $a_0\in\R$? (Or, equivalently, 
  is $\vfi$ non-uniformly
  expanding along the vertical direction?)
\item \label{prob:intermittent}
Consider $\toro, \alpha, \nu$ as in the previous item
\eqref{prob:vianamaps}. Let
$f_\theta(x)=f_1(x)+\theta \bmod1$ be a family of
  local diffeomorphisms of the circle, where $f_1$ is given in
  Example \ref{ex:intermittent}. Is
$\vfi(\theta,x)=(\alpha(\theta),f_{\theta_1}(x))$
  non-uniformly expanding for random orbits for some
  exponent $\beta\in(0,1)$? 
  We note that according to \cite{ArTah} the answer is
  affirmative if $f_1$ is defined with exponents
  $\beta\ge1$.
\item Consider $f_\theta$ as in the previous item
  \eqref{prob:intermittent}. Let
  $\alpha(\theta)=\theta+\omega\bmod1$ (for some fixed
  $\omega$) be an irrational rotation with uniquely ergodic
  measure $\nu=$ Lebesgue measure.  Is
  $\vfi(\theta,x)=(\alpha(\theta),f_{\theta}(x))$
  non-uniformly expanding along the vertical direction?
  What if we consider $\alpha$ a circle diffeomorphism with
  irrational rotation number and $\nu$ an ergodic
  $\alpha$-invariant probability which is non-atomic and
  singular with respect to Lebesgue measure?
  \end{enumerate}

\section{Basic invariant measures}
\label{subs:basic-invariant-measures}

We assume from now on that the skew-product map satisfies $(H_1)$,
$(H_2^*)$, $(H_3)$ and $(H_4)$ (or $(H_4^*)$) in the case $\Y=I_0$,
or it satisfies $(H_1)$, $(H_2^*)$, $(H_3)$, $(H_5)$ and $(H_6)$
in the case $\Y$ is other compact manifold. The condition $(H_2^*)$
is as follows
\begin{enumerate}				
\item[$(H_2^*)$] $\al:\toro\ra\toro$ is a bimeasurable
  bijection with an ergodic invariant probability measure
  $\nu$ such that $\nu(\Disc_\alpha)=0$ (we recall that
  $\Disc_\alpha$ is the set of discontinuity points of
  $\alpha$).
\end{enumerate}
In section
\ref{sec:Remove-H2} we show how to replace condition
$(H_2^*)$ by $(H_2)$.

We recall that $\leb$ is the normalized Lebesgue
measure on $I_0$. Since $\al$ is invertible, the
functions $f_{\al^{-j}(\te)}^{j}$ are well defined and
they send $\{\al^{-j}(\te)\}\times I_0$ on
$\{\te\}\times I_0$, for $\te\in\toro$, $j\geq 1$.
Thus, we can define the following measures on $I_0$,
for every $\te\in \toro$ and every $n\in\N$,
\begin{equation*}  \label{eq:eta-n-theta}
  \eta_n(\te)=\frac{1}{n} \sum_{j=1}^{n} (f_{\al^{-j}(\te)}^{j})_{*}
\leb
\end{equation*}
and using them, for every $n\in\N$ we define the following measures on
$\toro\times I_0$, 
\begin{equation*}
\eta_n= \int \eta_n(\te)\: d\nu(\te).
\end{equation*}
The integral above means that for any continuous function
$g:\toro\times I_0\to\R$ we have
\begin{align*}
  \eta_n(g)=\int g\,d\eta_n=\int \left( \int
    g(\theta,x)\,d\eta_n(\theta)(x) \right)d\nu(\theta).
\end{align*}

We recall from Section~\ref{sec:statem-results} that
$\mathcal{B}_{\toro}$ is the Borel $\sigma$-algebra on
$\toro$. To be able to define the measure $\eta_n$, we
need that for every continuous function $h:I_0\to\R$
the map
\begin{equation*}
\te\mapsto  \eta_n(\te)(h) = \int h \,d\eta_n(\theta)
\end{equation*}
is measurable. This is proved in Appendix~\ref{sec:measur}.

Assuming that these measures are all well-defined, we can
easily prove some key properties of the accumulation points
of $(\eta_n)_{n\ge1}$.

\begin{lemma}
  \label{le:marginal-nu}
  For every probability measure $\eta$ which is a weak$^*$
  limit of $(\eta_{n})_{n\ge1}$ we have that $\eta_n(A\times
  I_0)=\nu(A)$ for each $n\ge1$ and $\eta(A\times
  I_0)=\nu(A)$, for all $A\in\B_\toro$.
\end{lemma}

\begin{proof}
  We fix $A$ and $\eta$ as in the statement. Then we have
  for all $n\in\Z^+$ by definition $\eta_n(A\times
  I_0)=\int_A\eta_n(\theta)(I_0)\,d\nu(\theta)=\nu(A)$. If
  we take $A\in\B_\toro$ such that $\eta(\partial(A\times
  I_0))=\eta((\partial A) \times I_0)=0$, then using
  $\eta_{n_k}\xrightarrow[k\to+\infty]{w^*}\eta$ we get
  $\eta(A\times I_0)=\nu(A)$. Since the family of these sets
  generates $\B_\toro$ modulo $\eta$-null sets, we are done.
\end{proof}

\begin{lemma}
  \label{le:null-discont}
  For every probability measure $\eta$ which is a weak$^*$
  limit of $(\eta_{n})_{n\ge1}$ we have that
  $\eta(\Disc)=0$, where $\Disc$ is the set of discontinuity
  points of $\vfi$.
\end{lemma}

\begin{proof}
We consider the following cases.
\begin{description}
\item[Case 1] the maps $f_\theta$ are $C^3$ for all
  $\theta\in\toro$, that is, there are no discontinuities
  along the vertical direction: $\disc_\theta=\emptyset$ for
  all $\theta\in\toro$. Thus, it holds that
$\Disc\subset(\Disc_\alpha\times
  I_0)\cup(\Disc_{F}\times I_0)$. Then we have, by
  Lemma~\ref{le:marginal-nu}, that $\eta(\Disc)\le
  \eta(\Disc_\alpha\times I_0)+\eta(\Disc_F\times I_0)\le
  \nu(\Disc_\alpha)+\nu(\Disc_F)=0$ by $(H_2^*)$ and $(H_3)$.
\item[Case 2] we have discontinuities
  $\disc_\theta\neq\emptyset$ for some $\theta\in\toro$. But
  we assume that there is no recurrence to the set
  $\disc=\{(\theta,x): x\in\disc_\theta,
  \theta\in\toro\}$; 
Section~\ref{sec:statem-results}. 
  see condition $(H_4^*)$. Hence for every given
  $\ell\in\Z^+$ we can find an open neighborhood $V=V_\ell$
  of $\disc$ in $\toro\times I_0$ such that
  $\vfi^k(V)\cap V=\emptyset$ for all $k=1,\dots,\ell$.
  This implies that for any $z\in\toro\times I_0$ we have
  $\sum_{j=1}^n \chi_{V_\ell}(\vfi^j(z))\le (n/\ell)+1$. Thus,
  since $\eta(V)\le\liminf_{n\to+\infty}\eta_n(V)$ (see
  e.g. \cite[Theorem 2.1]{billingsley99}), it is enough to
  estimate for every big enough $n\in\Z^+$, using that $\nu$
  is $\alpha$-invariant and that $\alpha$ is invertible
  \begin{align*}
    \eta_n(V)
    &=
    \int\int
\frac1n\sum_{j=1}^n\chi_V\big(\theta,f^j_{\alpha^{-j}(\theta)}(x)\big)
    \,d\leb(x)\,d\nu(\theta)
    \\
    &=
    \int\frac1n\int\sum_{j=1}^n
    \chi_V\big(\vfi^j(\alpha^{-j}(\theta),x)\big)
    \,d\leb(x)\,d\nu(\theta)
    \\
    &=
    \int\frac1n\int\sum_{j=1}^n
    \chi_V\big(\vfi^j(\theta,x)\big)
    \,d\leb(x)\,d\nu(\theta)
    \le\frac2\ell.
  \end{align*}
  So for every $\ell>1$ we can find and open neighborhood
  $V$ of $\disc$ such that
  $\eta(\disc)\le\eta(V)\le2/\ell$.  Finally, since
  $\Disc\subseteq (\Disc_\alpha\times
  I_0)\cup(\Disc_{F}\times I_0)\cup\disc$ we obtain
  from the above together with Lemma~\ref{le:marginal-nu}
  \begin{align*}
    \eta(\Disc)
    &\le
    \eta(\Disc_\alpha\times I_0)+\eta(\Disc_{F}\times I_0)+
\eta(\disc)
    =\nu(\Disc_\alpha)+\nu(\Disc_F)=0
  \end{align*}
  as stated.
\item[Case 3] {In the higher dimensional setting, we have
  slow recurrence to the set of discontinuities $\disc\subset\cc$ 
  of $\vfi$ in the vertical direction}. Arguing by contradiction,
  let us assume that $\eta(\disc)>0$. Then there exists
  $a>0$ such that $\eta(B(\disc,\varrho))>a$ for all
  $\varrho>0$. 

  We fix $0<\varepsilon<a$ and then find $\delta>0$ given by
  the slow recurrence condition~\eqref{eq:slow-rec}. After
  that we fix $0<\varrho<\delta$ so that
  \begin{align*}
    \inf\{-\log\dist((\theta,x),\disc): (\theta,x)\in
    B(\disc,\varrho)\}>1
    \quad\text{and}\quad
    \eta_n(\partial B(\disc,\varrho))=0, \quad n\ge1
  \end{align*}
  and also $\eta(\partial B(\disc,\varrho))=0$. Then we note
  that, for each $n\ge1$, since $\nu$ is $\alpha$-invariant
  \begin{align*}
    a<\eta_n(B(\disc,\varrho))
    &=
    \int\int
\frac1n\sum_{j=1}^n\chi_{B(\disc,\varrho)}\big(\theta,f^j_{\alpha^{-j}
(\theta)}(x)\big)
    \,d\Leb(x)\,d\nu(\theta)
    \\
    &=
    \int\frac1n\int\sum_{j=1}^n
    \chi_{B(\disc,\varrho)}\big(\vfi^j(\alpha^{-j}(\theta),x)\big)
    \,d\Leb(x)\,d\nu(\theta)
    \\
    &=
    \int\frac1n\int\sum_{j=1}^n
    \chi_{B(\disc,\varrho)}\big(\vfi^j(\theta,x)\big)
    \,d\Leb(x)\,d\nu(\theta)
    \\
    &\le
    \int\int
    \frac1n\sum_{j=1}^n
-\log\dist_{\delta}\big(f^j_\theta(x),\crit_\theta\big)
    \,d\Leb(x)\,d\nu(\theta).
  \end{align*}
  Moreover, for big enough $n$ we get
  $\varepsilon>\eta_n(B(\disc,\varrho))\ge a$ thus
  $a<\varepsilon$. {This contradiction concludes the proof, since
$\Disc\subseteq (\Disc_\alpha\times
  I_0)\cup(\Disc_{F}\times I_0)\cup\disc$ as in Case 2}.
\end{description}  
\end{proof}

\begin{lemma}\label{le:every-weak-limit}
  Every weak$^*$ limit of $(\eta_{n})_{n\ge1}$ is a
  $\vfi$-invariant probability measure.
\end{lemma}
\begin{proof}
  Let us suppose, without loss of generality, that the
  sequence converges in the weak$^*$ topology to some
  probability measure, i.e., $\eta_{n} \ra \eta$ when
  $n\ra\infty$.  See Lemma \ref{le:tightness} and Remark
  \ref{rmk:tight-eta_n}.

  Let $g:\toro\times I_0\to\R$ be a continuous and bounded function.
We note that $\eta_n(g\circ\vfi)$ can be rewritten as
\begin{align*}
  \iint  g(\vfi(\theta,x)) \,d\eta_n(\te)(x) \:
  d\nu(\te) 
  &= \iint g(\alpha(\theta),f_\theta(x))
  \,d\eta_n(\te)(x)\: d\nu(\te)
  \\
  &=
  \int \left(\frac{1}{n}
    \sum_{j=1}^{n} (f_{\al^{-1}(\te)}\circ\ldots\circ
    f_{\al^{-j}(\te)})_{*}\leb\right)
  (g\circ\vfi)
  \: d\nu(\te)
  \\
  &=
  \frac{1}{n}
    \sum_{j=1}^{n}\!\!\iint\!\!\!
    g\Big(\alpha(\theta),
    f_\theta\big( f_{\al^{-1}(\te)}\circ\ldots\circ
    f_{\al^{-j}(\te)}(x)\big)\Big)
    \,d\leb(x)  \: d\nu(\te).
\end{align*}
But the last integral equals
\begin{align*}
  \int\frac{1}{n}\left(
    \int\sum_{j=1}^{n+1} 
    g(\alpha(\theta),
    (f_{\te}\circ f_{\al^{-1}(\te)}\circ\ldots\circ
    f_{\al^{-j+1}(\te)})(x))
    \,d\leb(x)
    -
    \int g(\alpha(\theta),f_{\te}(x))\,d\leb(x)
  \right) d\nu(\te)
\end{align*}
that is $ \int \big(\frac{n+1}{n} \eta_{n+1}(\al(\te))(g)-
\frac{1}{n}((f_{\te})_{*}\leb) (g(\alpha(\theta),\cdot)\big)
\: d\nu(\te)$.  We note that the last integral is bounded by
$\sup|g|$, which is finite.

Now since $\eta(\Disc)=0$ by Lemma~\ref{le:null-discont}, we
then arrive at (see e.g. \cite[Theorem 2.1]{billingsley99})
\begin{align*}
  (\fhi_{*} \eta)g
  &= 
  \lim_{n\ra\infty}\eta_n(g\circ \vfi)
  =
  \lim_{n\ra\infty} \int \frac{n+1}{n}\eta_{n+1}(\al(\te))(g) \:
d\nu(\te).
\end{align*}
But $\nu$ is $\alpha$-invariant and the function
$\theta\mapsto\eta_{n+1}(\al(\te))(g)$ is measurable, hence
the last expression equals
\begin{align*}
\lim_{n\ra\infty} &\int \frac{n+1}{n}\eta_{n+1}(\te)(g)
\: d\nu(\te)
= \lim_{n\ra\infty} \frac{n+1}{n}\eta_{n+1}(g)
= \eta(g).
\end{align*}
 This concludes the proof.
\end{proof}

\section{Absolutely continuous invariant
  measures} \label{sec:acim} 

Now we are going to define measures which are absolutely
continuous along the vertical fibers.  For this, we will use
the notion of hyperbolic-like times used in \cite{solano}.

\subsection{Notations and main technical result}
\label{sec:case-with-one}

We state a result for sequences of one dimensional
maps. This result is used to analyze the dynamics of the
skew-product restricted to the vertical leaves. Since we
have to consider skew-products in the different settings
$(H_4)$ and $(H_4^*)$, we also need to state the result for
sequences of one dimensional maps with conditions given by
these two settings. For $k\geq 0$, let us denote by
$\crit_k$ and $\disc_k$ the set of critical points and the
set of discontinuities, respectively, of $f_k:I_0\to I_0$.

We say that:
\begin{itemize}
\item \emph{a sequence of one dimensional maps} $\{f_k\}$
  \emph{satisfies} $(\widetilde{H_4})$ if: $f_k$ are $C^1$
  maps, $p:=\sup \{\# \crit_k, k\in\N \} <\infty$ and
  $\Gamma:=\sup \{ |f'_k(x)| , k\in\N, x\in I_0 \} <\infty $
  and the sequence $\{f'_k\}$ is equicontinuous.
\item \emph{a sequence of one dimensional maps}
$\{f_k\}$ \emph{satisfies} $(\widetilde{H_4^*})$ if: $f_k$ is a map
such that restricted to each connected component of $I_0\setminus
\disc_k$, is a $C^1$ diffeomorphism onto its image, $p:=\sup
\{\# \disc_k, k\in\N \} <\infty$ and $\Gamma:=\sup \{
|f'_k(x)| , k\in\N, x\notin \disc_k \} <\infty $. 
\end{itemize}
Finally we assume that for every $\ell\in\Z^+$, there exist
$\epsilon>0$ and neighborhoods $V_{\epsilon} \disc_k$ of
$\disc_k$ (for all $k\geq 0$) such that
\begin{equation} \label{eq:recurrence-property}
  f_i^j(V_{\epsilon}\disc_i)\cap V_{\epsilon}\disc_{i+j}=\emptyset
  \qquad \text{ for } i\geq 0, \:1\leq j\leq \ell.
\end{equation}
where $f_i^j=f_{i+j-1}\circ\ldots \circ f_{i+1}\circ f_{i}$.

Let us recall some additional definitions (see \cite{solano}
for more details). For every $x\in I_0$, $i\in\N$, we denote
  \begin{equation*}
  f^i(x):= f_{i-1}\circ\ldots\circ f_{1}\circ f_{0}(x),
  \end{equation*}
  and we write $T_i\left(\{f_k\},x\right)$ for  the
  \emph{maximal interval $T\subset I_0$, containing $x$ such
    that $f^i_{| T}$ is a $C^3$
    diffeomorphism}, and $r_i\left(\{f_k\},x\right)$ for the
  \emph{minimum between the lenghts of the connected
    components of $f^i(T_i(\{f_k\},x)\setminus \{x\})$}.
  

The following is a central technical result in our arguments. For the
proof see subsection \ref{sec:proof-of-theoremB}.

\begin{theorem} \label{PrincipalB} Let $\{f_k\}$ be a
  sequence of maps $f_k:I_0\ra I_0$ which satisfies
$(\widetilde{H_4})$
  (or $(\widetilde{H_4^*})$).  Assume that for some $\la>0$,
\begin{equation*}
  \displaystyle{\ilim \frac{1}{n} \log |Df^n(x)|>2\lambda}
  \quad 
\end{equation*}
for every $x\in E\subset I_0$. Then, there exists
$\constantB>0$ such that
\begin{equation} \label{liminfr_i}
\ilim \frac{1}{n} \soma r_i(\{f_k\},x)\geq 3\constantB 
\end{equation}
Lebesgue almost every $x\in E$.  Moreover, in the case of
$\{f_k\}$ satisfy condition $(\widetilde{H_4})$, $\constantB$ depends
only on $\la$, the modulus of continuity and the uniform
bound for the derivatives of the sequence $\{f_k\}$, and in
the uniform bound $p$ for the number of critical points. In
the case of $\{f_k\}$ satisfy condition $(\widetilde{H_4^*})$,
$\constantB$ depends only on $\la$, the uniform bound for
the derivatives of the sequence $\{f_k\}$ (outside of
discontinuities), the uniform bound $p$ for the number of
discontinuity points and the uniformity of $\epsilon$ on
condition \eqref{eq:recurrence-property}.
\end{theorem}

For our purposes the following sets are very useful:
\begin{equation*}
\begin{array}{rl}
  \mathcal{H}_i(\{f_k\},\secondelta)
  =&
  \left\{\:x\in I_0 ; \: r_i(\{f_k\},x)>\secondelta\right\};
  \\
  H_i(\{f_k\},\secondelta)
  =&
  \left\{\:x\in I_0 ; \: 
    r_i(\{f_k\},x)>\secondelta \text{ and }  
    |f^{i}(T_i(\{f_k\},x))|>3\secondelta \right\}.
 \end{array}
 \end{equation*}

 We will prove below that every connected component of
 $H_i(\{f_k\}, \secondelta)$ is sent diffeomorphically
 by $f^j$ onto its image with bounded distortion and
 the Lebesgue measure of the image is bounded away from
 zero.  We are interested in applying the last theorem
 to every sequence $\{f_{\al^{j}(\te)}\}_{j\in\Z^+}$,
 for each $\te\in\toro$. For simplicity, from now on,
 for $i\in\N$, $r_i(\te,x)$ denotes the set
 $r_i(\{f_k\},x)$, where $f_k=f_{\al^{k}(\te)}$ for
 every $k\ge0, \te\in\toro$. Analogously for the sets
 $T_i(\te,x)$, $\mathcal{H}_i(\te,\secondelta)$ and
 $H_i(\te,\secondelta)$.

We need the following result showing that
$(\widetilde{H_4^*})$ is a consequence of $(H_4^*)$.
\begin{lemma}
  \label{le:recurrence-H3-star}
  The above condition \eqref{eq:recurrence-property} is a
  consequence of the assumption $(H_4^*)$.
\end{lemma}
\begin{proof}
  We fix $\ell\in\Z^+$ and $V$ given by $(H_4^*)$. Consider
  also $i\geq 0$ and $1\leq j\leq \ell$.  We note that, by
  the skew-product nature of $\vfi$
  \begin{align*}
    \vfi^j\big(V\cap((\{\alpha^i(\theta)\}\times
    I_0)\big)\subset (\{\alpha^{i+j}(\theta)\} \times I_0)
    \cap \vfi^j(V).
  \end{align*}
  We now observe that the intersection
  in~\eqref{eq:recurrence-property} equals
  \begin{align*}
    \pi_2\Big(\vfi^j(V\cap(\{\alpha^i(\theta)\}\times
  I_0))\cap \big( V\cap(\{\alpha^{i+j}(\theta)\}\times
  I_0)\big)\Big)
    \subset
    \pi_2\big( (\{\alpha^{i+j}(\theta)\} \times I_0)
    \cap \vfi^j(V) \cap V \big)=\emptyset,
  \end{align*}
  where $\pi_2:\toro\times I_0\to I_0$ is the projection on
  the second coordinate. So we can use the neighborhoods $V$
  given by $(H_4^*)$ to obtain the neighborhoods
  $V_\epsilon\disc_i$ in \eqref{eq:recurrence-property}.
\end{proof}

\begin{remark}
  \label{rmk:uniform-epsilon}
  {The fact that $\epsilon>0$ does not depend
    on the sequence of maps chosen relies on the choice in
    $(H_4^*)$ of the neighborhood $V$ of the closure
    $\overline{\disc}$ of the set of
    discontinuities in $\toro\times I_0$.}
\end{remark}

We need the following result in the rest of the arguments.

\begin{lemma}[Pliss]
\label{le:pliss}
  Given $A \ge c_2>c_1>0$, let $\zeta=(c_2-c_1)/(A-c_1)$.
  Then, given any real numbers $a_1, \ldots, a_N$ such that
$$
\sum_{j=1}^{N} a_j \ge c_2 N
\qquad\text{and}\qquad
a_j \le A
\text{ for every }
1\le j \le N,
$$
there are $l > \zeta N$ and $1 < n_1 < \cdots < n_l \le N$
so that
$$
\sum_{j=n+1}^{n_i} a_j\ge c_1 (n_i-n)
\quad\text{for every } 0 \le n < n_i
\text{ and  }i=1, \dots, l.
$$
\end{lemma}

\begin{proof}
  See~\cite[Lemma 11.8]{Man87}.
\end{proof}

Thus, by the last theorem and using the Lemma of Pliss, we
have the following.

 \begin{corollary} \label{MainCorollary} 
 Let $\fhi: \toro\times I_0\ra \toro\times
   I_0$ be a skew-product as above satisfying $(H_1)$ and
   $(H_4)$ (or $(H_1)$ and $(H_4^*)$). Assume that there exists a set
   $E\subset \toro\times I_0$ and $\la>0$ such that
\begin{equation*}
 \ilim \frac{1}{n} \log |Df_\te^n(x)|> 2\la
\end{equation*}
for every $(\te,x)\in E$ and let us denote by $E(\te)$
the $\te$-section of the set $E$, that is,
$E(\theta)=\{x\in I_0 : (\theta,x)\in E\}.$ Then there
exist $\constantB>0$ and $\zeta>0$ such that for $n$ big
enough do not depend on $\te$
\begin{equation*}
  \int \frac{1}{n}\soma
  \leb\left(\mathcal{H}_i(\te,\constantB)\cap E(\theta)\right)
\,d\nu(\te) 
  \geq \frac{\zeta}{2}(\nu\times\leb)(E).
\end{equation*}
\end{corollary}

\begin{proof}
  Let us fix $\te\in\toro$ and consider the sequence
  $\{f_{\al^{j}(\te)}\}_{j\in\Z^+}$.  Let $\constantB>0$ be
  the constant found on Theorem \ref{PrincipalB}. 
  We consider the measure $\pi_n$ in $\{1,2, \ldots,n\}$
  defined by $\pi_n(B)=\#(B)/n$, for every subset $B$. Using
  Fubini's theorem, we have
\begin{equation*}
  \frac{1}{n} \soma
  \leb(\mathcal{H}_i(\te,\constantB)\cap E(\theta))
  =
  \int\int_{I_0}\bigchi(x,i) d\leb(x)d\pi_n(i)
  =
  \int_{I_0}\int \bigchi(x,i)d\pi_n(i)d\leb(x)
\end{equation*}
where $\bigchi(x,i)=1$ if $x\in
\mathcal{H}_i(\te,\constantB)\cap E(\theta)$ and $\bigchi(x,i)=0$
otherwise. Applying Pliss Lemma~\ref{le:pliss}, we conclude
the existence of $\zeta>0$ such that $\int
\bigchi(x,i)d\pi_n(i)\geq \zeta$ if $x$ is such that $x\in
E(\theta)$ and $\soma r_i(\te,x)\geq 2\constantB n$. Hence
\begin{equation*}
  \frac{1}{n} \soma
  \leb(\mathcal{H}_i(\te,\constantB)\cap E(\theta) )
\geq \:\zeta \leb\,\left(\left\{\,x\in E(\theta);
\:\soma r_i(\te,x)\geq 2\constantB n\right\} \right).
\end{equation*}
By Theorem \ref{PrincipalB}, we have that
\begin{equation*}
 \leb\,\left(\left\{ x\in E(\te) : 
     \soma r_i(\te,x)
     \geq
     2\constantB n, \text{ for all }n\geq N\right\} \right)
 \to 
 \leb(E(\te) )
\end{equation*}
when $N\to \infty$.  Since the constant $\constantB$ is the
same for any sequence $\{f_{\al^{j}(\te)}\}_{j\in\Z^+}$,
varying $\te\in \toro$, the result follows using the
Dominated Convergence Theorem.
\end{proof}

For any $\secondelta>0$, if $r_i(\{f_k\},x)>2\secondelta$ then $
|f^{i}(T_i(\{f_k\},x))|>4\secondelta$. Thus,
$\mathcal{H}_i(\{f_k\},2\secondelta)\subset
H_i(\{f_k\},\secondelta)\subset
\mathcal{H}_i(\{f_k\},\secondelta)$. Therefore, we get a similar
result to the last corollary for $H_i$ instead of $\mathcal{H}_i$.

\subsection{Construction of absolutely continuous invariant
  probability measures}
Assume that we are in the conditions of Theorem
\ref{mthm:PrincipalA}.  Clearly, the set $Z$ in the
statement of Theorems \ref{mthm:PrincipalA} and
\ref{mcor:finiteergodicmeasures} may be taken positively
invariant under $\fhi$. Given any $\la>0$, let $Z(\la)$ be
the set of points in $Z$ for which the inequality
\eqref{eq:posiexponents} holds. Then $Z(\la)$ is positively
invariant. As usual, $Z(\te,\la)$ denotes the
$\te$-section of the set $Z(\la)$.

Let us fix a constant $\la>0$.  Let
$\constantB>0$ be the constant found on Theorem
\ref{PrincipalB}. 
Thus, we define the following measures on
$I_0$, for every $n\in\N$ and $\te\in \toro$
\begin{equation} \label{eq:mu-n-theta}
  \mu_n(\te)=\frac{1}{n} \sum_{j=1}^{n}
  (f_{\al^{-j}(\te)}^{j})_{*} \big(\leb \mid
  Z(\alpha^{-j}(\theta),\la)\cap
H_{j}(\al^{-j}(\te),\constantB)\big).
\end{equation}

 Using these measures, for every $n\in\N$ we
  define the following on $\toro\times I_0$,
\begin{equation}\label{eq:mu_n=int-mu_n-theta}
\mu_n= \int \mu_n(\te)\: d\nu(\te)
\end{equation}

Again  we need to show that for every continuous
  function $h:I_0\to\R$ the map
\begin{equation*}
\te\mapsto  \mu_n(\te)(h) = \int h \,d\mu_n(\theta)
\end{equation*}
is measurable. This is proved in Appendix~\ref{sec:measur}.

\begin{lemma}
  \label{le:tightness}
  For all $n\ge1$ and $A\in\B_\toro$ we have $\mu_n(A\times
  I_0)\le\nu(A)$. Moreover, this conditions ensures that the
  sequence $(\mu_n)_{n\ge1}$ of measures is tight in
  $\toro\times I_0$; thus it is relatively compact in the
  weak$^*$ topology of measures in $\toro\times I_0$.
\end{lemma}

\begin{proof}
  We just observe that $\mu_n(A\times I_0)=\int_A
  \mu_n(\theta)(I_0)\,d\nu(\theta)$ by definition, and also
  $\mu_n(\theta)(I_0)\le1$ for each $\theta$.  In addition,
  from this property and the assumption that $\mu_n$ are
  Borel measures on $\toro$ which is a separable metrizable
  and complete topological space, given $\epsilon>0$ we can fix a
  compact subset $\toro_0\subset\toro$ such that
  $\nu(\toro\setminus\toro_0)<\epsilon$ and we obtain
  \begin{align*}
    \mu_n(\toro\times I_0\setminus(\toro_0\times I_0))
    &=
    \mu_n((\toro\setminus\toro_0)\times I_0)\le
    \nu(\toro\setminus\toro_0) < \epsilon
  \end{align*}
  uniformly in $n\ge1$, as required for tightness of the
  family $(\mu_n)_{n\ge1}$. We can now apply Prokhorov's
  Theorem to obtain the final conclusion of the statement of
  the lemma; see~\cite[Chapter 1, Section 5]{billingsley99}. 
\end{proof}

\begin{remark}
  \label{rmk:tight-eta_n}
  Lemma~\ref{le:marginal-nu} together with the previous
  arguments also shows that $(\eta_n)_{n\ge1}$ is a tight
  sequence of probability measures in $\toro\times I_0$.
\end{remark}

Now we obtain the absolute continuity of $\mu_n(\theta)$
with respect to $\leb$.

\begin{lemma}\label{le:bdd-dens}
There exists $K>0$ such that for any measurable subset $A\subset I_0$,
\begin{equation*}
\mu_n(\te)(A)\leq K\leb(A)
\end{equation*}
for every $\te\in\toro, n\in\N$. Moreover, $K$ depends only
on the constant $\constantB$ in the definition of
$H_i(\te,\constantB)$.
\end{lemma}

\begin{proof}
  Let $J$ be a connected component of
  $H_j(\al^{-j}(\te),\constantB)$. Let us consider the maximal 
interval $T$, containing $J$, such that
  $f_{\al^{-j}(\te)}^{j}$ restricted to $T$ is a $C^3$
  diffeomorphism. 
  There exists $\tau$, depending
  only on $\constantB$, such that $f_{\al^{-j}(\te)}^{j}(T)$
  contains a $\tau$-scaled neighborhood of
  $f_{\al^{-j}(\te)}^{j}(J)$ (i.e., both connected
  components of $f_{\al^{-j}(\te)}^{j}(T)\setminus
  f_{\al^{-j}(\te)}^{j}(J)$ have length $\geq \tau
  |f_{\al^{-j}(\te)}^{j}(J)|$). By Koebe Principle (see \cite[Theorem
IV.1.2]{MS93}),
  $f_{\al^{-j}(\te)}^{j}$ restricted to $J$ has bounded
  distortion (by a constant $K\p$ depending only on
  $\constantB$). On the other hand, $|f_{\al^{-j}(\te)}^{j}(J)|$ is
  bounded away from zero. Thus, we conclude that
  $(f_{\al^{-j}(\te)}^{j})_{*} \big(\leb \mid
Z(\alpha^{-j}(\theta),\la)\cap
  H_{j}(\al^{-j}(\te),\constantB)\big)(A)\leq K\leb(A)$ for any
measurable
  set $A\subset I_0$.
\end{proof}

From the previous lemma we deduce the absolute continuity of
$\mu_n$ with respect to $\nu\times\leb$.

\begin{lemma}\label{le:bdd-dens-mu}
  Let $K>0$ be as in Lemma~\ref{le:bdd-dens}. Then for
  any $W\in \mathcal{B}_{\toro}\times\mathcal{B}_{I_0}$
  we have $\mu_n(W)\leq K\cdot(\nu\times \leb)(W) $ for
  all $n\in\N$.
\end{lemma}

\begin{proof}
  The set $\mathcal{A}=\{W\in
  \mathcal{B}_{\toro}\times\mathcal{B}_{I_0}; \mu_n(W)\leq K\cdot
  (\nu\times\leb)(W) \}$ is a $\sigma$-algebra.  On the other
  hand, if $W=F\times A$ for some $F\in
  \mathcal{B}_{\toro}$, $A\in \mathcal{B}_{I_0}$, we
  conclude, from the definition of $\mu_n$ and the last claim,
  that $W\in\mathcal{A}$. This is enough to conclude the proof.
\end{proof}

Now we extend the results of the previous lemmas to the
cluster points of the sequence $\mu_n$ in the weak$^*$ topology.

\begin{lemma}\label{le:bdd-dist-limit}
  Let $K>0$ be as in Lemma~\ref{le:bdd-dens}. Then for
  any weak$^*$ limit $\mu$ of $\{\mu_n\}_n$, we have
  $\mu(W)\leq K\cdot(\nu\times\leb)(W)$ for any $W\in
  \mathcal{B}_{\toro}\times\mathcal{B}_{I_0}$.
\end{lemma}

\begin{proof}
  The set $\mathcal{A}=\{W\in
  \mathcal{B}_{\toro}\times\mathcal{B}_{I_0}; \mu(W)\leq K\cdot
  (\nu\times\leb)(W) \}$ is a $\sigma$-algebra.  Since
  $\mu_{n_k}$ converges in the weak$^*$ topology to $\mu$,
  $\mu(W)\leq \liminf \mu_n(W)$ for any open
  set $W$. Also note that from
  Lemma~\ref{le:bdd-dens-mu}, for open sets $W\in
  \mathcal{B}_{\toro}\times\mathcal{B}_{I_0}$,
  $\mu_n(W)\leq K\cdot(\nu\times\leb)(W)$ for all $n\in\N$.
  As these sets generate
  $\mathcal{B}_{\toro}\times\mathcal{B}_{I_0}$, the claim
  follows.
\end{proof}
By definition $\mu_n$ is a part of the measure $\eta_n$, for
all $n\in\N$. Let $\xi_n$ be a measure such that
\begin{equation}\label{eq:eta_n}
\eta_n=\mu_n+\xi_n 
\end{equation}
for all $n\in\N$. From Lemma~\ref{le:tightness} and
Remark~\ref{rmk:tight-eta_n} we assume, without loss of
generality, that there exist some subsequence $\{n_k\}_k$
and measures $\eta,\mu,\xi$ such that 
$\eta_{n_k},\mu_{n_k},\xi_{n_k}$ converge to
$\eta,\mu,\xi$, when $k\to\infty$, respectively.  We then
have
\begin{align}\label{eq:eta=limit}
  \eta=\mu+\xi.
\end{align}

Let $\beta_1$ and $\beta_2$ be measures on the same
measurable space. As usually, if $\beta_1$ is absolutely
continuous with respect to $\beta_2$, we write
$\beta_1\ll\beta_2$; and if $\beta_1$ is singular with
respect to $\beta_2$, we write $\beta_1\perp\beta_2$.

Next we show that the Lebesgue decomposition of an invariant
measure with respect to any finite measure $\measure$, for a
non-singular transformation, is formed by invariant measures.

\begin{lemma}
  \label{le:inv-Leb-decomp}
  Let us assume that a measurable transformation
  $T:(X,\mathcal{X}) \to (X, \mathcal{X})$ satisfies
  $T_*\measure\ll\measure$ for some finite measure $\measure$
  in $(X,\mathcal{X})$ (that is, $T$ is non-singular with
  respect to $\measure$). We assume also that a $T$-invariant
  probability measure $\eta$ is given with Lebesgue
  decomposition $\eta=\mu+\xi$, with $\mu\ll\measure$ and
  $\xi\perp\measure$.  Then both $\mu$ and $\xi$ are
  $T$-invariant measures.
\end{lemma}

\begin{proof}
  Since $\xi\perp\measure$, we may find $E\in\mathcal{X}$
  such that $\measure(E)=0$ and $\xi(X\setminus E)=0$. In
  particular, $\xi(A)=\xi(A\cap E)$ for all
  $A\in\mathcal{X}$. Because $\measure(E)=0=\measure(T^{-1}(E))$ we get
  \begin{align*}
    \xi(T^{-1}(E))
    =\mu(T^{-1}(E))+\xi(T^{-1}(E))
    =\eta(T^{-1}(E))
    =\eta(E)
    =\mu(E)+\xi(E)=\xi(E)
  \end{align*}
  and $E$ is $T$-invariant $\xi\bmod0$, i.e., $\xi(E
  \bigtriangleup T^{-1}(E))=0$. Hence
  $\xi(X\setminus T^{-1}(E))=0$ and $\measure(T^{-1}(E))=0$.
  Thus for $A\in\mathcal{X}$
  \begin{align*}
    \xi(T^{-1}(A))=\xi(T^{-1}(A)\cap
    T^{-1}(E))=\xi(T^{-1}(A\cap E))
    =\xi(A\cap E)=\xi(A)
  \end{align*}
  since $\xi(T^{-1}(A\cap E))=\eta(T^{-1}(A\cap
  E))=\eta(A\cap E)=\xi(A\cap E)$. We have proved that $\xi$
  is $T$-invariant.

  Therefore
  \begin{align*}
    \mu+\xi=\eta=T_*\eta=T_*\mu+T_*\xi=T_*\mu+\xi
  \end{align*}
  shows that $T_*\mu=\mu$ and $\mu$ is also $T$-invariant
\end{proof}

\subsubsection{Existence of absolutely continuous invariant
probability
measure}
\label{sec:existence-absolut-co}

Now we use the previous results to complete the proof of
existence of an absolutely continuous invariant probability
measure for $\vfi$. The ergodicity is proved
in Section~\ref{sec:finitely-many-ergodi}.
\begin{proposition} \label{pr:existenceofacim} Let $\fhi:
  \toro\times I_0\ra \toro\times I_0$ be a skew-product as
  above satisfying $(H_1)$, $(H_2)$, $(H_3)$ and $(H_4)$ (or
  $(H^{*}_4)$). Assume that $\nu\times\leb(Z(\la))>0$. Then
  there exists an absolutely continuous invariant
  probability measure which gives positive mass to $Z(\la)$.
\end{proposition} 

\begin{proof}
  Let us consider the measures $\mu_n$ given by
\eqref{eq:mu_n=int-mu_n-theta}. We recall that by~\eqref{eq:eta_n}
  and~\eqref{eq:eta=limit} we have $\eta=\mu+\xi$ with
  $\mu\ll\nu\times\leb$.  By the Lebesgue Decomposition Theorem,
  there exist (unique) measures $\xi^{ac}$ and $\xi^{s}$
  such that $\xi^{ac}\ll \nu\times\leb$,
  $\xi^{s}\perp\nu\times\leb$ and
  $\xi=\xi^{ac}+\xi^{s}$. Then we have a decomposition of
  $\eta=(\mu+\xi^{ac})+\xi^{s}$ as a sum of one absolutely
  continuous measure and a singular one (both with respect
  to $\nu\times\leb$). On the other hand, notice that $\fhi_*
  (\nu\times\leb)\ll \nu\times\leb$ (it follows from the
  invariance of $\nu$ and by the non-singularity of
  $f(\te,\cdot)$, for every $\te\in\toro$).
  
  The previous Lemma~\ref{le:inv-Leb-decomp} ensures that
  $\mu+\xi^{ac}$ is an absolutely continuous
  $\vfi$-invariant measure.

  We claim that 
  $\mu+\xi^{ac}$ is a non-zero finite measure.
  It suffices to prove that there exists $\gamma>0$ such
  that, $\mu_n(\toro\times I_0)>\gamma$ for all $n$ big
  enough.
  Using that $\al^{-1}$ is invariant by $\nu$ and defining
  the family $s_j(\theta):=\leb(H_j(\theta,\constantB)\cap
  Z(\te,\la))$ of measurable functions for $j\ge1$, we have
  for all $n\in\Z^+$,
\begin{align*}
  \mu_n(\toro\times I_0)
  &=
  \int_{\toro}\frac{1}{n}\sum_{j=1}^{n}
  \leb\big(H_j(\al^{-j}(\te),\constantB)\cap
Z(\alpha^{-j}(\theta),\la)\big)
\:d\nu(\te) =
  \frac{1}{n}\sum_{j=1}^{n} \int_{\toro} s_j\circ
  \al^{-j}(\te)\: d\nu(\te)\\
  &= \frac{1}{n}\sum_{j=1}^{n} \int_{\toro} s_j(\te) \:
  d\nu(\te)
  = \int_{\toro} \frac{1}{n}\sum_{j=1}^{n}
\leb(H_j(\te,\constantB)\cap
Z(\te,\la))\: d\nu(\te).
\end{align*}
By Corollary~\ref{MainCorollary} this last integral is
bounded away from zero, as long as the set $Z(\la)$
has positive $\nu\times\leb$-measure. More precisely, we
have $\mu_n(Z(\la))\ge \frac\zeta2 (\nu\times\leb)(Z(\la))$
for all big enough $n$.  Hence the normalization of
$\mu+\xi^{ac}$ satisfies the conditions of the statement.
\end{proof}



\subsection{Proof of the technical result}
\label{sec:proof-of-theoremB}
Here we present a proof of Theorem \ref{PrincipalB}.
The proof in the setting $(\widetilde{H_4^*})$ is
similar to the proof on the setting $(\widetilde{H_4})$. The result on
the setting $(\widetilde{H_4})$ corresponds to Theorem B in
\cite{solano}, but here we do not assume the equicontinuity
of the sequence $\{f_k\}$. For completeness, we prove the result on
the setting $(\widetilde{H_4^*})$ and we remark the modifications for
the proof on the setting $(\widetilde{H_4})$.


\subsubsection{Definitions and fundamental lemmas}
In order to simplify the notation we say that $f^j(x)\in V_\ep \disc$
if $f^j(x)\in V_\ep \disc_j$ for $j\in\N$. By the recurrence property 
on the setting $(\widetilde{H_4^*})$ (see equation
(\ref{eq:recurrence-property})) we have that 
\begin{lemma} \label{lemmafrequence}
Given $\gamma>0$, there exists $\ep>0$ such that for $n$ big enough,
\begin{equation} \label{eq:frequenceofvisits}
 \frac{1}{n}\sum_{j=0}^{n-1} \bigchi_{V_\ep \disc}
(f^j(x))<
\gamma
\end{equation}
for any $x\in I_0$.
\end{lemma} 
\begin{remark}
Note that the lemma also holds on setting $(\widetilde{H_4})$. In this
case, (\ref{eq:frequenceofvisits}) holds for any $x$ such that $\log
|Df^n(x))|>\la n$ and $\epsilon$ depends on $\la$. We use the
equicontinuity of the sequence $\{f'_k\}$ instead of condition
(\ref{eq:recurrence-property}).
\end{remark}
On the other hand, since the derivative of the maps of the sequence
$\{f_k\}$ is bounded from above outside of the set of discontinuities,
it holds the following result. 
\begin{lemma} \label{lemmaboundsize}
Given $\ep>0$ and $l\in\N$, there exists $\de>0$ such that for any
subinterval $J\subset I_0$,
\begin{equation*}
 \text{if } \: |J|\leq 2\de \:\text{ and }\: f_i^j(J)\cap
\disc_{i+j}=\emptyset \:\text{ for } 0\leq j< k \qquad\text{then}
\qquad|f_i^k(J)|<\ep 
\end{equation*}
for all $i\geq 0$ and $0< k\leq l$.
\end{lemma}
\begin{proof}
Let us consider $\Gamma=\sup \{|Df_k(x)|; k\in\N, x\notin \disc_k\}$.
The lemma follows from the next claim: for all $i,k\geq 0$, for any
interval $J\subset I_0$, if $f_i^j(J)\cap \disc_{i+j}=\emptyset$ for 
$0\leq j< k$, then  $|f_i^k(J)|\leq \Gamma^k|J|$.
\end{proof}

\begin{remark}
  Notice that this lemma also holds on the setting
  $(\widetilde{H_4})$, replacing the set $\disc$ by $\crit$.
\end{remark}
The main part in the proof of Theorem \ref{PrincipalB} is the control
of the Lebesgue measure of the sets $Y_n(\la)\cap
A_n\left(\{f_k\},\de\right)$, where $Y_n(\la)=\{x, \log |Df^n(x)|>\la
n\}$ and
\begin{equation*}\label{Antheta}
A_n(\de)=A_n\left(\{f_k\},\de\right):= \Biggl\{ x\in I_0\:;\:
\frac{1}{n}\soma r_i(\{f_k\},x)<\de^2, \:\:r_n(\{f_k\},x)>0\Biggr\},
\end{equation*}
for $n\in\N$ and $\de>0$ (and $r_i$ as was defined 
just before the statement of Theorem \ref{PrincipalB}). 
For simplicity, we denote by $A_n(\de)$ the
set $A_n\left(\{f_k\},\de\right)$ and by $r_i(x)$ the number
$r_i(\{f_k\},x)$. 

We introduce the following sets.
For $\de>0$, $a_i\in \{0,1\}$ for $i=1, 2, \ldots, n$,
\begin{equation*} 
C_{\de}(a_1,a_2,\ldots,a_n):=\{x\in  I_0\:;
\: r_i(x)\geq\de\:\: \text{if}\:\: a_i=1, 
\:\,0<r_i(x)<\de \:\:\text{if}\:\: a_i=0
\}
\end{equation*}
Note that for every $x\in A_n(\de)$, there exist $a_1,\ldots, a_n$
(with $a_i\in \{0,1\}$ for $i=1,\ldots,n$) and $J$ component of
$C_{\de}(a_1,a_2,\ldots,a_n)$ such that $x\in J$.

The key lemma in the proof of Theorem \ref{PrincipalB} is the
following. Let $\# X$ denotes the number of connected
components of $X$.
\begin{lemma} \label{lemmaprin}
Given $\la>0$, there exists $\de>0$ such that  $\sum \#
C_\de(a_1,\ldots,a_n)\leq \exp (n\la/2)$, where the sum is over all
$a_1,\ldots, a_n$ such that $\sumaupla{n}<\de n$. Moreover, the
dependence of $\de$ is as $\constantB$ on the statement of Theorem
\ref{PrincipalB}.
\end{lemma}


We need to decompose the interval $I_0$ set in a convenient
way. Given $\ep>0$, $m\leq n$, $\{t_1,\ldots,t_m\}\subset
\{0,1,\ldots,n-1\}$, we define
\begin{equation*}
K_{n,\ep}(t_1,\ldots,t_m)=\{x\in I_0;\: f^j(x)\in V_{\ep}\disc \text{
if and only if } j\in\{t_1,\ldots,t_m\}\}
\end{equation*} 
By Lemma \ref{lemmafrequence} we conclude that given $\ga>0$, there
exists $\ep>0$ such that for $n$ big enough,
\begin{equation*} \label{decompositionYn}
I_0=\cup_{m=0}^{\ga n}\cup_{t_1,\ldots, t_m} K_{n,\ep}(t_1,\ldots,t_m)
\end{equation*} 
where the second union is over all subsets $\{t_1, \ldots,
t_m\}\subset\{0,1,\ldots, n-1\}$.

Let us denote by $\# \{I\subset C_{\de}(a_1,\ldots,a_n); I\cap
K_{n,\ep}(t_1,\ldots,t_m)\neq\emptyset\}$ the number of 
connected components of
$C_{\de}(a_1,\ldots,a_n)$ whose intersection with
$K_{n,\ep}(t_1,\ldots,t_m)$ is non empty.

From the last equation 
we conclude that
\begin{equation} \label{A_n<Cfinal} 
\sum_{a_1,\ldots,a_n} \# C_\de(a_1,\ldots,a_n)\leq
  \sum_{a_1,\ldots,a_n} \sum_{t_1,\ldots, t_m} \# \{I\subset
  C_{\de}(a_1,\ldots,a_n); I\cap
  K_{n,\ep}(t_1,\ldots,t_m)\neq\emptyset\}
\end{equation}
where the first sum is over all $a_1,\ldots, a_n$ such that
$a_1+\ldots+a_n <\de n$ and the second sum is over all subsets $\{t_1,
\ldots, t_m\}\subset \{0,1,\ldots, n-1\}$ with $m<\ga n$. 
\begin{remark}
In the setting $(\widetilde{H_4})$, we count the number of components
of $C_{\de}(a_1,\ldots,a_n)$ whose intersection with $Y_n(\la)$ is non
empty. In order to do it, instead of the sets
$K_{n,\ep}(t_1,\ldots,t_m)$, we use the sets
$Y_{n,\ep}(t_1,\ldots,t_m):=Y_n(\la)\cap K_{n,\ep}(t_1,\ldots,t_m)$.
\end{remark}

\subsubsection{Components of 
$C_\de(a_1,\ldots, a_s)$}
We state some claims related to the number of connected 
components of the sets
$C_\de(a_1,\ldots, a_n)$. Recall that $p$ is the maximum number of
elements in any $\disc_k$ (for $k\geq 0$). Given $I\subset I_0$ and
$s\in\N$, we say $f^s(I)\cap\disc=\emptyset$ (resp. $\neq\emptyset$)
if $f^s(I)\cap \disc_{s}=\emptyset$ (resp. $\neq\emptyset$).
\begin{claim}\label{3(p+1)timesthelast}
For any $a_1,a_2,\ldots, a_s$ with $a_j\in\{0,1\}$ for all $j$,
\begin{equation*}
\#C_{\de}(a_1,\ldots,a_s,0)+\#C_{\de}(a_1,\ldots,a_s,1)\leq
3(p+1)\#C_{\de}(a_1,\ldots,a_s)
\end{equation*}
\end{claim}
\begin{proof}
Let $I$ be a component of $C_\de(a_1,\ldots, a_s)$. If $f^s(I)\cap
\disc=\emptyset$ and $I\p\subset I$ is a component of
$C_\de(a_1,\ldots, a_s,0)$, it can not exist one component of
$C_\de(a_1,\ldots, a_s,1)$ at each side of $I\p$. So, there exist at
most two components of  $C_\de(a_1,\ldots, a_s,0)$ in $I$. Hence, $I$
is divided at most  in 3 components of $C_\de(a_1,\ldots, a_s,0)\cup
C_\de(a_1,\ldots, a_s,1)$. 

If $f^s(I)\cap \disc\neq\emptyset$, I is divided at most in $p+1$
components. Each one of these components have a boundary which goes by
$f^s$ to $\disc$ and is divided (as for the last case) at most  in 3
components of $C_\de(a_1,\ldots, a_s,0)\cup C_\de(a_1,\ldots, a_s,1)$.
\end{proof}

\begin{claim}\label{withoutsingularity}
Let $s, n\in\N$ and $J$ be a component of $C_{\de}(a_1,\ldots,a_s,0)$.
If $f^{s+i}(J)\cap \disc=\emptyset$ for $1\leq i\leq n$, then
\begin{equation*}
\#\{I\subseteq C_{\de}(a_1,\ldots,a_s, 0^{i+1}), I\subseteq J\}\leq
i+1.
\end{equation*}
for $1\leq i\leq n$, where $0^{i+1}$ means that the last $i+1$ terms
are equal to 0.
\end{claim}
\begin{proof}
For $i=1$ the proof is contained on the proof of Claim
\ref{3(p+1)timesthelast}.  
Now, note that every connected component of 
$C_{\de}(a_1,\ldots,a_s, 0^{i})$
gives rise to one or two components of $C_{\de}(a_1,\ldots,a_s,
0^{i+1})$.
The proof of Claim \ref{withoutsingularity} follows by induction on
$i$, showing that at most one component of $C_{\de}(a_1,\ldots,a_s,
0^{i})$ gives rise to two components of $C_{\de}(a_1,\ldots,a_s,
0^{i+1})$. 
\end{proof}

To bound the number of connected components whose intersection with
$K_{n,\ep}(t_1,\ldots,t_m)$ is non-empty, we have the following claim.
\begin{claim}\label{withsingularity}
Let $l\in \N$ and $\ep>0$ be constants and let $\de=\de(l)$ be the
number given by Lemma \ref{lemmaboundsize}. For any $a_1,\ldots, a_s$
with $a_j\in\{0,1\}$, $\{t_1,\ldots, t_m\}\subset \{0,1,\ldots,n-1\}$.
If $\{s+1,\ldots,s+i\}\cap \{t_1,\ldots, t_m\}=\emptyset$ and $i\leq
l$, then
\begin{multline*}
\#\{I\subseteq  C_{\de}(a_1,\ldots,a_s,0^{i+1}), I\cap
K_{n,\ep}(t_1,\ldots,t_m)\neq\emptyset \} \leq (i+1) \#\{I\subseteq
C_{\de}(a_1,\ldots,a_s,0), I\cap
K_{n,\ep}(t_1,\ldots,t_m)\neq\emptyset \}.
\end{multline*}
\end{claim}
\begin{proof}
Let $I$ be a component
 of $C_\de(a_1,\ldots,a_s,0)$. Then $|f^{s+1}(I)|\leq
 2\de$. Let $i_0\in\{1,2,\ldots, i\}$ the first number such
 that $f^{s+i_0}(I)\cap \disc\neq\emptyset$. Since
 $f^{s+j}(I)\cap \disc=\emptyset$ for $0\leq j<i_0$, by Lemma
 \ref{lemmaboundsize}, $|f^{s+i_0}(I)|<\ep$. Then, for all
 $x\in I$, $f^{s+i_0}(x)\in V_\ep\disc$. Since
 $\{s+1,\ldots,s+i\}\cap \{t_1,\ldots, t_m\}=\emptyset$, then
 $I\cap K_{n,\ep}(t_1,\ldots,t_m)=\emptyset$. Hence, if
 $I\cap K_{n,\ep}(t_1,\ldots,t_m)\neq\emptyset$ and
 $\{s+1,\ldots,s+i\}\cap \{t_1,\ldots, t_m\}=\emptyset$, then
 $f^{s+j}(I)\cap \disc=\emptyset$ for all $0\leq j\leq i$,
 with $i\leq l$. Thus, claim follows using Claim
\ref{withoutsingularity}.
\end{proof}
\subsubsection{Proof of Lemma \ref{lemmaprin}}

Given $m<n$, $\de>0$ and $\ep>0$, let us consider $a_1,\ldots,a_n$
with $a_i\in\{0,1\}$ (such that $\sumaupla{n}<\de n$) and 
$\{t_1,\ldots, t_m\}\subset \{0,\ldots,n-1\}$. 
We can decompose the sequence $a_1\ldots a_n$ in maximal blocks of
0\textquoteright s and 1\textquoteright s. We write the symbol $\xi$
in the $j$-th position if $a_j=1$ or, $a_j=0$ and $j=t_k$ for some 
$k\in\{1,\ldots,m\}$. In this way we have,  
\begin{equation} \label{blockdecompo}
a_1 a_2\ldots a_n=\xi^{i_1}0^{j_1}\xi^{i_2}0^{j_2}\ldots
\xi^{i_h}0^{j_h}
\end{equation}
with $0\leq i_k,j_k\leq n$ for $k=1,\ldots,h$,
$\sum_{k=1}^{h}(i_k+j_k)=n$ and $\sum_{k=1}^{h} i_k < m+\de n$. 

Lets us assume that $a_1,\ldots, a_n$ are as in (\ref{blockdecompo}).
Let $l,\ep$ and $\de$ be as in Lemma \ref{lemmaboundsize}.  Using
claims \ref{3(p+1)timesthelast}  and \ref{withsingularity} we have,
\begin{align*}
\#\{I \subset &C_{\de}(a_0,\ldots,a_n), I\cap
K_{n,\ep}(t_1,\ldots,t_m)\neq\emptyset\}\leq \\
&\leq (3(p+1)(l+1)^{\frac{j_h}{l}+1}(3(p+1))^{i_h}) \ldots
(3(p+1)(l+1)^{\frac{j_1}{l}+1}(3(p+1))^{i_1}) \\
&\leq (3(p+1))^{\sum_{k=1}^{h}i_k} (3(p+1))^h
(l+1)^{\frac{\sum_{k=1}^{h}j_k}{l}+h} \\
&\leq (3(p+1))^{m+\de n+h} (l+1)^{\frac{n}{l}+h}.
\end{align*}

Therefore, if $\sumaupla{n}< \de n$ and $m<\ga n$ we conclude from the
inequality above that for $n$ big enough,
\begin{equation}
\begin{split} \label{fundamental}
\#\{I&\subset C_{\de}(a_1,\ldots,a_n), I\cap
K_{n,\ep}(t_1,\ldots,t_m)\neq\emptyset\} \\
&\leq (3(p+1))^{\ga n+\de n}(3(p+1))^{2(\de+\ga)n}
(l+1)^{\frac{n}{l}+2(\de+\ga)n} \leq \exp (n \:\psi_0(l,\ga,\de) )
\end{split}
\end{equation}
where $\psi_0(l,\ga,\de)=3(\de+\ga)\log (3(p+1))+
2(\de+\ga+\frac{1}{l})\log (2l)$. 

Using (\ref{fundamental}) and Stirling\textquoteright s formula in
equation (\ref{A_n<Cfinal}), we conclude that

\begin{equation*}
\sum_{a_1,\ldots, a_n} \# C_\de(a_1,\ldots,a_n)\leq \exp (n\:
\psi_3(l,\ga,\de))
\end{equation*}
where
$\psi_3(l,\ga,\de)=\psi_0(l,\ga,\de)+ \psi_1(\ga)+\psi_2(\de)$,
$\psi_1(\ga)\ra 0$ and $\psi_2(\de)\ra 0$ when $\ga\ra 0$ and $\de\ra
0$, respectively. 
Hence, we have to choose $l$ such that
\begin{equation} \label{choosel}
\frac{2}{l}\log (2l)<\frac{\la}{14}
\end{equation}
and, let $\ga>0$ be such that
\begin{equation} \label{choosega}
2\ga\log (2l)<\frac{\la}{14}, \qquad
3\ga \log (3(p+1))<\frac{\la}{14}, \quad \text{ and } \quad
\psi_1(\ga)<\frac{\la}{14}
\end{equation}
Next, we find $\ep>0$, using Lemma \ref{lemmafrequence}. Finally,
given $\ep$ and $l$, let $\de>0$ be the constant given by Lemma
\ref{lemmaboundsize} and satisfying 
\begin{equation} \label{choosede}
2\de \log (2l)<\frac{\la}{14}, \qquad
3\de\log (3(p+1))<\frac{\la}{14} \quad \text{ and } \quad
\psi_3(\de)<\frac{\la}{14}
\end{equation}
With this choice, $\psi_3(l,\ga,\de)\leq \frac{\la}{2}$. Hence the
first part of Lemma \ref{lemmaprin} is proved. On the other hand,
observe that the choice of $\de$ is given fundamentally by Lemmas
\ref{lemmafrequence} and \ref{lemmaboundsize}. Namely, $\de$ depends
on: the constant $\la$ in the definition of $Y_n(\la)$; the uniformity
of $\ep$ (given $\ell\in\N$) on the equation
(\ref{eq:recurrence-property}); the uniform boundedness of $|Df_k|$ on
the proof of Lemma \ref{lemmaboundsize}; and the uniform boundedness
of the number of discontinuity points for $f_k$, where $k\geq 0$. This
concludes the proof of Lemma \ref{lemmaprin}. \qed
\subsubsection{Proof of Theorem \ref{PrincipalB}}
Note that for every $N\in\N$ it holds
\begin{equation*}
E\cap\left(\bigcap_{n\geq N} Y_n(\la)\right) \cap
\complement\left(\bigcup_{n\geq N} A_n(\de)\cap Y_n(\la)\right)
\:\mbox{\Large $\subset$}\:E\cap\left(\bigcap_{n\geq N} \complement
A_n(\de)\cap Y_n(\la)\right). 
\end{equation*}
where $\complement B$ denotes the complement set of $B$ and $|B|$
denotes the Lebesgue measure of $B$. By the hypotheses of theorem, 
$|E\cap\left(\cap_{n\geq N} Y_n(\la)\right) |$ converges to the
Lebesgue measure of $E$. On the other hand, note that if $J$ is a
component of $C_\de(a_1,\ldots, a_n)$ (with $a_1+\ldots+a_n<\de n$)
then $|J\cap Y_n(\la)\cap A_n(\de)|\leq |I_0| \exp (-n\la)$. Then,
using Lemma \ref{lemmaprin} we conclude that $|\cup_{n\geq N}
A_n(\de)\cap Y_n(\la)|$ converges to zero when $N\to\infty$. 
Therefore, $|\cap_{n\geq N} (\complement A_n(\de)\cap Y_n(\la))\cap
E|$ converges to $|E|$ when $N\to\infty$. Thus, we conclude that
(\ref{liminfr_i}) holds considering $3\constantB=\de^2$. \qed

\section{Non-invertible base transformation}
 \label{sec:Remove-H2}

 Let $\fhi: \toro\times I_0\ra \toro\times I_0$ or
 $\varphi:\toro\times\Y\to\toro\times\Y$ be a skew-product
 satisfying $(H_2)$ and the remaining conditions of
 Theorems~\ref{mcor:finiteergodicmeasures} or
 \ref{mthm:PrincipalA}.  We define a natural extension
 $\hat{\fhi}$ of this map and we prove that it satisfies
 $(H_2^*)$ and also the remaining conditions of the
 statement of the Main Theorems.

\subsection{Inverse limit construction}
\label{sec:inverse-limit-constr}

We use a standard construction which
allows to define, for an endomorphism of a measure space, an
induced invertible bimeasurable map of a new measure
space. For more details, see for instance \cite[Chapter
10.4]{CoFoSi82}. We perform the construction with the map
$\al:\toro\ra \toro$.

First consider the (inverse limit) space $\hat{\toro}$ which
is formed by points
$$\hat{\te}=(\te_{0},\te_{-1},\te_{-2},\ldots),$$ 
where $\te_{-i}\in\toro$ for $i\geq 0$ and
$\al(\te_{-i})=\te_{-i+1}$ for $i\geq 1$.
Then we have 
\begin{enumerate}
\item $\toro^{\N}$ with the product topology is a metrizable
  space (see \cite[Lemma 111.15]{kasriel});
\item $\toro^{\N}$ is separable (see \cite[Theorems 111.14
  and 58.7]{kasriel});
\item as a topological space (in fact, a metrizable space),
  $\toro^{\N}$ admits the Borel $\sigma$-algebra
  $\mathcal{B}_{\toro^{\N}}$, which is the $\sigma$-algebra
  generated by the open sets of the product topology on
  $\toro^{\N}$;
\item the product $\sigma$-algebra $\prod_{i\in\N}
  \mathcal{B}_{\toro_i}$ on $\toro^{\N}$ coincides with
  $\mathcal{B}_{\toro^{\N}}$. ($\toro_i=\toro$ for all
  $i\in\N$);

\item as subset of $\toro^{\N}$, $\hat{\toro}$ is endowed with
  the product topology, and therefore has a Borel
  $\sigma$-algebra $\mathcal{B}_{\hat{\toro}}$;
\item $\mathcal{B}_{\hat{\toro}}$ coincides with the
  $\sigma$-algebra obtained by intersecting $\prod_{i\in\N}
  \mathcal{B}_{\toro_i}$ with $\hat{\toro}$.
\end{enumerate}



Now, $\hat{\toro}$ with the $\sigma$-algebra
$(\prod_{i\in\N} \mathcal{B}_{\toro_i})\cap \hat{\toro}$ is
a measurable space. For the sets of the form
\begin{equation*}
 (A)_n=
 \{\hat{\te}=(\te_{0},\te_{-1},\te_{-2},\ldots)\in \hat{\toro}; \:
\te_{-n}\in A\}
\end{equation*}
where $A\in \mathcal{B}_{\toro}$ and $n\geq 0$, we define
$\hat{\nu}((A)_n)=\nu(A)$. Since these sets generate the
$\sigma$-algebra and the conditions of compatibility of
Kolmogorov\s s Theorem are satisfied, we have a measure
$\hat{\nu}$ defined on the $\sigma$-algebra.

We can consider the map
$\hat{\al}:\hat{\toro}\ra\hat{\toro}$ given by
\begin{equation*}
  \hat{\al}((\te_{0},\te_{-1},\te_{-2},\ldots))
  =(\al(\te_{0}),\al(\te_{-1}),\al(\te_{-2}),\ldots)
  =(\al(\te_{0}),\te_{0},\te_{-1},\te_{-2},\ldots).
\end{equation*}
This map is invertible ${\hat{\al}}^{-1}
((\te_{0},\te_{-1},\te_{-2},\ldots))=(\te_{-1},\te_{-2},\te_{-3}
\ldots)$. The
measure $\hat{\nu}$ is invariant with respect to
$\hat{\al}$. 

Therefore we have constructed an invertible map $\hat{\al}$,
bimeasurable (with the Borel $\sigma$-algebra
$\mathcal{B}_{\hat{\toro}}$) on a metric space
$\hat{\toro}$, such that
$\pi_{0}\circ\hat{\al}(\hat{\te})=\al\circ
\pi_{0}(\hat{\te})$ for every $\hat{\te}\in \hat{\toro}$,
where $\pi_{0}(\hat{\te})=\te_0$. It is also useful to define
the natural projection map $P:\hat{\toro}\times I_0\ra \toro\times
I_0$, by
$P(\hat{\te},x)=(\pi_{0}(\hat{\te}),x)=(\te_{0},x)$.

\subsection{Non-invertible base}
\label{sec:case-with-one-2}

Let us define the map $\hat{\fhi}:\hat{\toro}\times I_0\ra
\hat{\toro}\times I_0$, 
$\hat{\fhi}(\hat{\te},x)=(\hat{\al}(\hat{\te}),
\hat{f}(\hat{\te},x))$, where
$\hat{f}(\hat{\te},x)=f(\te_{0},x)$. Since $\hat{\al}$, $P$
and $f$ are measurable 
then $\hat{\fhi}$ is measurable, i.e,
${\hat{\fhi}}^{-1}(\mathcal{B}_{\hat{\toro}}\times
\mathcal{B}_{I_0})\subset \mathcal{B}_{\hat{\toro}}\times
\mathcal{B}_{I_0}$.

Note that the set of critical and discontinuity points for
$\hat{f}_{\hat{\te}}$ projects onto the corresponding set
for $f_{\te_0}$. Hence the measurability of the set
\begin{equation*}
  \hat{\cc}
  =
  \{ (\hat{\te},x)\in \hat{\toro}\times I_0 :
  x\in \crit_{\theta_0}\cup\disc_{\theta_0} \}
  =
  P^{-1}(\cc)
\end{equation*}
follows from the measurability of the set $\cc$ and of the
map $P$. Thus, $\hat{\fhi}$ satisfies
condition $(H_1)$.

We note that the set of discontinuity points $\Disc_{\hat{\alpha}}$
of $\hat{\alpha}$ coincides with the set
\begin{equation*}
  (\Disc_{\alpha})_0=\{ \hat{\te}\in \hat{\toro}; \te_0\in
\Disc_{\alpha}\}
\end{equation*}
and so
$\hat{\nu}(\Disc_{\hat{\alpha}})=\nu(\Disc_{\alpha})=0$. On
the other hand, for the map $\hat{F}:\hat{\toro}
\to B(I_0), \hat\theta\mapsto \hat{f}_{\hat{\theta}}=f_{\theta_0}$
we have that 
$\Disc_{\hat{F}}\subset (\Disc_{F})_{0}$. 
Hence the map $\hat{\fhi}$ satisfies conditions $(H
_2^*)$ and
$(H_3)$. The map $\hat{\fhi}$ clearly satisfies condition
$(H_4)$ (resp. $(H_4^*)$) if the map $\fhi$ satisfies the
condition $(H_4)$ (resp. $(H_4^*)$).

Moreover, if $\fhi$ is non-uniformly expanding along
the vertical direction according to $\nu\times\leb$, on the subset
$Z$, then $\hat{\fhi}$ is non-uniformly expanding along
the vertical direction according to $\hat{\nu}\times\leb$, on the
subset $P^{-1}Z$. It also holds that
$\hat{\nu}\times\leb(P^{-1}Z)=\nu\times\leb(Z)$.

Thus, $\hat{\fhi}$ is a skew-product in the conditions of Theorems
\ref{mcor:finiteergodicmeasures} and \ref{mthm:PrincipalA}.

We remark that in order to prove the relative compactness of the
sequences of measures $\{\eta_n\}$ and $\{\mu_n\}$ (see Lemma
\ref{le:tightness} and Remark \ref{rmk:tight-eta_n}) we use the fact
that $\toro$ is a separable metrizable and complete topological
space. The space $\hat{\toro}$ can fail to be complete. To solve this
problem, we can consider $\hat{\nu}$ as a measure defined on
$\toro^\N$ (stating that $\hat{\nu}(\toro^\N\setminus\hat{\toro})=0$).
Thus, we can find a closed set $\toro_1\subset
\hat{\toro}$ such that
$\hat{\nu}(\hat{\toro}\setminus\toro_1)<\epsilon$. On the other hand,
since $\toro^\N$ is a separable metrizable and complete topological
space, we can find a compact set $\toro_2\subset \toro^\N$, such that
$\hat{\nu}(\toro^\N\setminus\toro_2)<\epsilon$. Hence, for the
compact set $\toro_1\cap\toro_2$, we have that
$\hat{\nu}(\hat{\toro}\setminus(\toro_1\cap\toro_2))<2\epsilon$.
Therefore, considering $\toro_0=\toro_1\cap\toro_2$ in Lemma
\ref{le:tightness} and Remark \ref{rmk:tight-eta_n}, the relative
compactness of the sequences of measures $\{\eta_n\}$ and $\{\mu_n\}$
follows.

Hence we may repeat the same sequence of steps in the
arguments in Section~\ref{sec:acim} assuming
Theorem~\ref{PrincipalB} to conclude the result in
Proposition \ref{pr:existenceofacim}: there exists an
invariant probability measure $\hat{\mu}$ which is
absolutely continuous with respect to $\hat{\nu}\times
\leb$, with $\hat{\mu}(P^{-1}(Z(\la)))>0$. Now we push this
measure for the original space $\toro\times I_0$.
\begin{lemma}\label{le:invariance-hat}
  $P_* \hat{\mu}$ is an $\fhi$-invariant probability measure which is
  absolutely continuous with respect to $\nu\times\leb$, and $P_*
\hat{\mu}(Z(\la))>0$.
\end{lemma}

\begin{proof}
  Let $A\subset \toro\times I_0$ be a measurable
  subset. Using that $\fhi\circ P=P\circ \hat{\fhi}$ and the
  $\hat{\fhi}$-invariance of $\hat{\mu}$ we have that
\begin{equation*}
P_*\hat{\mu}(\fhi^{-1}(A))
=\hat{\mu}(P^{-1}\fhi^{-1}(A))
= \hat{\mu}((P\circ \hat{\fhi})^{-1}(A))
=\hat{\mu}(\hat{\fhi}^{-1}(P^{-1}A))
=P_*\hat{\mu}(A)
\end{equation*}
and then $P_* \hat{\mu}$ is invariant with respect to
$\fhi$. 

On the other hand, if $(\nu\times\leb)(A)=0$, then
$(\hat{\nu}\times\leb)(P^{-1}A)=0$. Using the absolute
continuity of $\hat{\mu}$, we conclude that $P_*
\hat{\mu}(A)=0$.
\end{proof}

\section{Finitely many ergodic basins }
 \label{sec:finitely-many-ergodi}


 Here we conclude the proofs of Theorems \ref{mcor:finiteergodicmeasures} and \ref{mthm:PrincipalA}, proving that the
 invariant sets with positive $\nu\times\leb$-measure, have
 mass bounded away from zero.

 Given $\lambda>0$, let $Z(\lambda)\subset \toro\times I_0$
 the set of points with vertical Lyapunov exponents greater
 than $2\lambda$, i.e., points in $Z$ for which the
 inequality (\ref{eq:posiexponents}) holds.

\begin{proposition}\label{pr:finitelymanyergodic}
  Given $\lambda>0$, there exists $b>0$ such that every
$\vfi$-invariant subset $G\subset
Z(\lambda)$ with positive
  $\nu\times\leb$-measure satisfies
  $(\nu\times\leb)(G)>b$.
\end{proposition}

This ensures that the ergodic basins $B_i=B(\mu_i)$ of the
measures provided by
Theorem~\ref{mcor:finiteergodicmeasures} has
$\nu\times\leb$-measure uniformly bounded away from
zero. Since these are pairwise disjoint subsets, their
number in a finite measure space must be finite.

For the proof of Proposition, we need the following result.
\begin{lemma} \label{le:verticaldistortion} Given
  $\constantB>0$, there exists $K_1>0$ such that, for any
  $i\in\N$ and any $(\te,x)\in \toro\times I_0$, if
  $r_i(\te,x)>\constantB$, there exists $J_i(x)\subset I_0$ such
  that
  $f_\te^i(J_i(x))=B(f_\te^i(x),\constantB/2)$,
  $f_\te^i$ restricted to $J_i(x)$ is a $C^3$
  diffeomorphism and
\begin{equation} \label{eq:verticaldistortion}
 \frac{1}{K_1}\leq \frac{|Df_\te^i(y)|}{|Df_\te^i(z)|}\leq K_1
\qquad\text{ for all } y,z\in J_i(x).
\end{equation}
\end{lemma}
\begin{proof}
Let $(\te,x)\in \toro\times I_0$ such that $r_i(\te,x)>\constantB$.
By definition of $r_i$, there exists $T_i\subset I_0$ such that $x\in
T_i$, $f_\te^i$ restricted to $T_i$ is a $C^3$ diffeomorphism and the
connected components of $f_\te^i(T_i)\setminus \{f_\te^i(x)\}$ have
length $> \constantB$. Let us choose $J_i(x)\subset T_i$ such that
$f_\te^i(J_i(x))=B(f_\te^i(x),\constantB/2)$. Note that
$f_\te^i(T_i)$ contains an $1/2$-scaled neighborhood of
$f_\te^i(J_i(x))$. It means that both connected components of
$f_\te^i(T_i)\setminus f_\te^i(J_i(x))$ have length at least
$|f_\te^i(J_i(x))|/2$. By Koebe Principle (see \cite[Theorem
IV.1.2]{MS93}), there exists $K_1$ such that
(\ref{eq:verticaldistortion}) holds. The distortion $K_1>0$ does not
depend on the point $(\te,x)$, nor the iterate $i$.
\end{proof}

\begin{proof}[Proof of Proposition \ref{pr:finitelymanyergodic}]
Let $G\subset Z(\lambda)$ be a forward $\vfi$-invariant set, such
that $(\nu\times\leb)(G)>0$. Given $\lambda>0$, let us consider the
constant $\constantB>0$ given by Theorem \ref{PrincipalB}. Let $K_1>0$
the constant found on Lemma \ref{le:verticaldistortion}. Denoting by
$G(\te)$ the $\te$-section of $G$, i.e, $G(\te):=\{x\in I_0;
(\te,x)\in G \}$, let us define the measurable set
\begin{equation*}
B_{G}^{\constantB}:=\left\{\te\in \toro,
\:\leb\left(G(\te)\right)\geq 
 \frac{\constantB}{4K_1} \right\}
\end{equation*}
Since the measure $\nu$ is ergodic for the map $\al$, then
\begin{equation} \label{eq:measureofB_G}
 \lim_{n\to\infty} \frac{1}{n}\sum_{i=0}^{n-1}
\bigchi_{B_{G}^{\constantB}}(\al^i(\te))=\int
\bigchi_{B_{G}^{\constantB}} d\nu
\end{equation}
for all $\te$ in a $\nu$-full measure set. Let $\te_0\in\toro$ be a
point such that $\leb(G(\te_0))>0$ and (\ref{eq:measureofB_G})
holds for $\te=\te_0$. By Theorem \ref{PrincipalB} applied to the
set $E=G(\te_0)$, we can find a point $x_0\in
G(\te_0)$ such that $\liminf_{n\to\infty} \frac{1}{n}\sum_{i=1}^{n}
r_i(\te_0,x_0)\geq 3\constantB$. We can assume that $x_0$ is
a density point of $G(\te_0)$. Thus, there exists $\epsilon_0>0$ such
that for all $\epsilon<\epsilon_0$,
\begin{equation*}
 \frac{\leb(G(\te_0)\cap
B(x_0,\epsilon))}{\leb(B(x_0,\epsilon))}\geq \frac{1}{2}
\end{equation*}
On the other hand, we can find $N\in\N$ such that $\sum_{i=1}^{n}
r_i(\te_0,x_0)\geq 2\constantB n$ and $\log |Df_{\te_0}^n(x_0)|\geq
\la n$ for all $n\geq N$. Then, if $r_i(\te_0,x_0)\geq \constantB$,
for $i\geq N$, the interval $J_i(x_0)$ found on Lemma
\ref{le:verticaldistortion} is such that $|J_i(x_0)|\leq K_1\constantB
e^{-\la i}$. Therefore, there exists $n_0\geq N$ such that
$J_i(x_0)\subset B(x_0,\epsilon_0)$, provided $i\geq n_0$ (and
obviously only when $r_i(\te_0,x_0)\geq \constantB$). 

\begin{claim}
If $r_i(\te_0,x_0)\geq \constantB$ and $i\geq n_0$ then
$\al^i(\te_0)\in  B_{G}^{\constantB}$.
\end{claim}
\begin{proof}
Let $J_i^*(x_0)$ the maximal ball centered at
$x_0$ contained in $J_i(x_0)$. Using Lemma
\ref{le:verticaldistortion},
\begin{equation*}
\frac{\leb(f_{\te_0}^i(J_i^*(x_0)\cap
G(\te_0)))}{\leb(f_{\te_0}^i(J_i^*(x_0)))}\geq
\frac{1}{K_1}\frac{\leb(J_i^*(x_0)\cap
G(\te_0))}{\leb(J_i^*(x_0))}\geq \frac{1}{2K_1}
\end{equation*}
for $i\geq n_0$. Then we have that $\leb(f_{\te_0}^i(J_i^*(x_0)\cap
G(\te_0))\geq \constantB/4K_1$. Since $f_{\te_0}^i(G(\te_0))\subset
G(\al^i(\te_0))$ (by the forward $\fhi$-invariance of $G$), the
claim follows. 
\end{proof}
An immediate consequence of the claim is that for all $n\geq n_0$,
\begin{equation*}
 \sum_{i=n_0}^{n}\bigchi_{B_{G}^{\constantB}}(\al^i(\te_0))\geq
\#\:\{\,n_0\leq i\leq n ; \, r_i(\te_0,x_0)\geq\constantB\}
\end{equation*}

Now, using Pliss Lemma (see Lemma \ref{le:pliss}), there exists
$\zeta=\zeta(\constantB)>0$ such that for $n\geq n_0$,
\begin{equation*}
\frac{\#\:\{\,1\leq i\leq n ; \, r_i(\te_0,x_0)\geq
\constantB\}}{n}\geq \zeta
\end{equation*}
since $\sum_{i=1}^{n} r_i(\te_0,x_0)\geq 2\constantB n$, for $n\geq
n_0$. Hence the limit in (\ref{eq:measureofB_G}) for $\te=\te_0$ is
greater than $\zeta$. It means that $\nu(B_G^{\constantB})\geq \zeta$.
Thus, Proposition \ref{pr:finitelymanyergodic} follows considering
$b=\constantB\zeta/4K_1$.
\end{proof}

Finally we can conclude the proof of Theorems
\ref{mcor:finiteergodicmeasures} and \ref{mthm:PrincipalA}.

\begin{proof}[Proof of Theorem
  \ref{mcor:finiteergodicmeasures}]

  Assume that $(\nu\times\leb)(Z)>0$ (otherwise, there is
  nothing to prove).  By assumption, there exists $\la>0$
  such that $Z\setminus Z(\la) $ has zero
  $(\nu\times\leb)$-measure. Let $\mu_0$ be the
  $\fhi$-invariant probability measure absolutely continuous
  with respect to $\nu\times\leb$ given by Proposition
  \ref{pr:existenceofacim} and Lemma
  \ref{le:invariance-hat}.  Considering the normalized
  restriction to the forward invariant set $Z(\la)$, we can
  assume that $\mu_0(Z(\la))=1$. Since every invariant set,
  with positive $\nu\times\leb$-measure, has
  $\nu\times\leb$-measure greater than $b$ (by Proposition
  \ref{pr:finitelymanyergodic}), we can decompose $\mu_0$ in
  a finite number of ergodic components. Then
  $\mu_0=\sum_{i=1}^{s} a_i\mu_i$, where $a_i>0$,
  $\sum_{i=1}^{s} a_i=1$ and $\mu_i$ are ergodic
  $\fhi$-invariant absolutely continuous probability
  measures.

  If $Z_1=Z(\la)\setminus \cup_{i=1}^{s} B(\mu_i)$ still has
  positive $\nu\times\leb$-measure, then we can repeat the
  arguments of Section \ref{sec:acim} for the set
  $Z_1\subset Z(\lambda)$ instead of $Z(\la)$.  Repeating
  this argument, we obtain the ergodic components as in the
  statement of Theorem \ref{mcor:finiteergodicmeasures}
  such that $\nu\times\leb$-a.e. point in $Z(\la)$ is in the
  basin of one of these measures. The number of such measures
  is finite, since the basin of each of then is a collection
  of pairwise dijoint invariant sets with
  $\nu\times\leb$-positive measure, and Proposition
  \ref{pr:finitelymanyergodic} holds.
\end{proof}

\begin{proof}[Proof of Theorem \ref{mthm:PrincipalA}]
  Since $Z=\cup_{n\in \N} Z(1/n)$, the previous argument
  applied to each $Z(1/n)$ provides finitely many ergodic
  probability measures whose basins cover $Z(1/n)$, for each
  $n\ge1$. This concludes the proof of Theorem
  \ref{mthm:PrincipalA}.
\end{proof}

\section{SRB measures for random non-uniformly expanding
  maps}
\label{sec:SRBrandommap}

Let $(\toro, \mathcal{B}_{\toro},\nu, \al,f)$ be an admissible random
non-uniformly expanding map on $I_0$. Let us consider the associated
skew-product $\fhi$ defined on $\toro\times I_0$. 
By Theorem \ref{mcor:finiteergodicmeasures}, there exist $\mu_1,
\ldots, \mu_t$, $\fhi$-invariant ergodic probabilities, such that
$(\nu\times\leb)$-a.e. $(\te,x)$ is in the basin
of one of these measures. Denote by $B_i$ the ergodic basin $B(\mu_i)$
of the measure $\mu_i$, for $1\leq i\leq t$. As usual,
$B_i(\te)$ denotes the $\te$-section of the set $B_i$.

\begin{proof}[Proof of Theorem \ref{mthm:finitelymanySRB}]
Define $\prob_i$ as the projection on $I_0$ of $\mu_i$. By a
straightforward calculation, we can prove that
$RB_\te(\prob_i)\supseteq B_i(\te)$.
As $\mu_i$ is absolutely continuous with respect to $\nu\times\leb$,
then $\nu\times\leb(B_i)>0$. Since $B_i$ is $\fhi$-invariant and
$\nu$ is $\al$-ergodic, then $\leb(B_i(\te))>0$ for $\nu$-almost
every $\te\in\toro$. It implies that $\prob_i$ is a SRB probability
for the random dynamical system.

Since $\leb\left( I_0\setminus \cup_{i=1}^{t}
B_i(\te)\right)=0$, for $\nu$-almost every $\te\in\toro$, then
$\nu$-almost surely, the union of the random basins of
$\prob_1,\ldots, \prob_t$ has total Lebesgue measure. Clearly, these
measures are absolutely continuous with respect to Lebesgue measure.
\end{proof}


\section{Higher dimensional fibers}
\label{sec:higher-dimens-fibers-0}

Here we outline the arguments in the higher-dimensional
fiber case. The strategy is the same as the one presented
for one-dimensional fibers. 

We start by considering the sequences $\eta_n(\theta)$ and
$\eta_n$ as in Section~\ref{subs:basic-invariant-measures}.
Then we use the notion of hyperbolic times from~\cite{ABV00}
to redefine $\mu_n(\theta)$ replacing
$H_n(\theta,\constantB)$ by $H_n(\theta)$.  Finally we just
have to obtain the analogous results to
Corollary~\ref{MainCorollary} and Lemmas~\ref{le:bdd-dens},
~\ref{le:bdd-dens-mu} and~\ref{le:bdd-dist-limit}.

After this the argument  follows the proof of
Theorem~\ref{mcor:finiteergodicmeasures} through
Lemmas~\ref{le:bdd-dist-limit} and \ref{le:inv-Leb-decomp}.

In what follows, since hyperbolic times have been
extensively investigated recently, we cite most of the
results from other published works.

\subsection{Hyperbolic times and their properties}
\label{sec:hyperb-times-their}

Given $0<\sigma<1$ and $b,\delta>0$, we say that the positive
integer $n$ is a $(\sigma,\delta,b)$-hyperbolic time for
$(\theta,x)\in\toro\times\Y$ if
\begin{align}
  \label{eq:tempo-hip}
  \prod_{j=n-k}^{n-1}\big\|
  Df_{\alpha^j(\theta)}\big(f_\theta^j(x)\big)^{-1}\big\|
  \le \sigma^k 
  \quad\mbox{and}\quad 
  \dist_\delta \big(
  f_\theta^k(x),\cc\cap(\{\alpha^k(\theta)\}\times\Y) \big) \ge
  e^{-bk} \quad\mbox{for}\quad k=0,\dots,n-1.
\end{align}

We now outline the properties of these special times. For detailed
proofs see~\cite[Lemma 5.2, Corollary 5.3]{ABV00} and~\cite[Proposition
2.6, Corollary 2.7, Proposition 5.2]{AA03}.

\begin{proposition}
  \label{pr:prophyptimes}
  There are constants $C_1,\delta_1>0$ depending on
  $(\sigma,\de,b)$ and $\vfi$ only such that, if $n$ is
  $(\sigma,\de,b)$-hyperbolic time for $(\theta,x)$, then
  there are neighborhoods $V_n(\theta,x)$ of $(\theta,x)$ on
  $\{\theta\}\times\Y$, such that
\begin{enumerate} 
\item $f^n_{\theta}\mid V_n(\theta,x)$ maps
  $V_n(\theta,x)$ diffeomorphically to the ball of radius
  $\delta_1$ around $f_\theta^n(x)$ inside
  $\{\alpha^n(\theta)\}\times\Y$;
\item 
$\dist\big(f_{\theta}^{n-k}(y),f^{n-k}_{\theta}(z)\big)\le
  \sigma^{k/2}\cdot
  \dist\big(f_\theta^{n}(y),f_\theta^{n}(z)\big)$
for every $0\leq k\leq n-1$ and $y,z\in V_n(\theta,x)$;
\item for $y,z\in V_n(\theta,x)$
\[
\frac1{C_1}\le \frac{\big|\det
  Df^{n}_{\theta}(y)\big|}{\big|\det
  Df_{\theta}^{n}(z)\big|} \le C_1.
\]
\end{enumerate}
\end{proposition}

The following ensures existence of infinitely many
hyperbolic times for Lebesgue almost every point for
non-uniformly expanding maps with slow recurrence to the
singular set. A complete proof can be found in~\cite[Section
5]{ABV00}.

\begin{theorem}
\label{thm:tempos-hip-existem}
Let $\vfi:\toro\times\Y\to\toro\times\Y$ be as in the
statement of Theorem~\ref{mthm:PrincipalB}, i.e.,
non-uniformly expanding along the fibers according to
$\nu\times\Leb$, on a subset $Z$ of $\toro\times\Y$.

Then there are $\sigma\in(0,1)$, $\delta,b>0$ and there
exists $\rho=\rho(\sigma,\delta,b)>0$ such that
$\nu\times\Leb$-a.e.  $(\theta,x)\in Z$ has
infinitely many $(\sigma,\de,b)$-hyperbolic times.  Moreover
if we write $0<n_1<n_2<n_2<\dots$ for the hyperbolic times
of $(\theta,x)\in Z$, then their asymptotic frequency satisfies
\[
\liminf_{N\to\infty}\frac{\#\{ k\ge1 : n_k\le
  N\}}{N}\ge\rho
\quad\mbox{for}\quad \nu\times\Leb\mbox{-a.e.  }
(\theta,x)\in Z.
\]
\end{theorem}

Now we define, in this setting
\begin{align*}
  \H_n(\sigma,\delta,b):=\{(\theta,x)\in\toro\times\Y: n
  \text{  is a $(\sigma,\delta,b)$-hyperbolic time for  } (\theta,x)\}
\end{align*}
and, having fixed $\sigma,\delta,b$ according to
Theorem~\ref{thm:tempos-hip-existem}, we set
\begin{align*}
  H_n(\theta) := (\{\theta\}\times\Y)\cap\H_n(\sigma,\delta,b).
\end{align*}

\subsection{Hyperbolic times on fibers}
\label{sec:hyperb-times-fibers}

Now we are able to state and prove the analogous result to
Corollary~\ref{MainCorollary} with the same arguments.
\begin{lemma}
  \label{le:positive-mass-infinity}
  Given $\lambda>0$, let $Z(\lambda)\subset Z\subset\toro\times\Y$ be
  such that $(\theta,x)\in Z(\lambda)$ satisfies
\begin{align*}
  \limsup_{n\to+\infty}\frac1n\sum_{j=0}^{n-1}
  \log\|Df_{\alpha^j(\theta)}(f_\theta^j(x))^{-1}\|<- 2\lambda<0.
\end{align*}
If $n$ is big enough we have $ \int \frac{1}{n}\soma
\Leb\left(H_i(\te)\right) \,d\nu(\te) \geq
\frac{\rho}2(\nu\times\Leb)(Z(\lambda))$, where $\rho>0$ is
given by Theorem~\ref{thm:tempos-hip-existem}.
\end{lemma}

We assume the measurability of $H_n(\theta)$ in what
follows. This will be proved in Appendix~\ref{sec:measur}.


We define the  measures $\mu_n(\theta)$ on $\Y$, for every $\te\in
\toro$ and every $n\in\N$, adapting \eqref{eq:mu-n-theta} as
follows
\begin{equation*}
  \mu_n(\te)=\frac{1}{n} \sum_{j=1}^{n}
  (f_{\al^{-j}(\te)}^{j})_{*} \big(\Leb \mid
  Z(\alpha^{-j}(\theta),\lambda)\cap
H_{j}(\al^{-j}(\te))\big).
\end{equation*}
and then we define the measures $\mu_n$ on $\toro\times\Y$
as in~\eqref{eq:mu_n=int-mu_n-theta}.
We need to show that these measures are well-defined and
again this is proved in Appendix~\ref{sec:measur}.



\begin{lemma}\label{le:bdd-dens-hyptimes}
  There exists $K>0$ such that $\mu_n(\te)(A)\leq
  K\cdot\Leb(A)$ for any measurable subset $A\subset\Y$ and
  every $\te\in\toro, n\in\N$.
\end{lemma}

The proof of this result 
follows~\cite[Proposition 5.2]{AA03}.

\begin{proof}[Proof of Lemma~\ref{le:bdd-dens-hyptimes}]
  Take $\delta_1>0$ given by
  Proposition~\ref{pr:prophyptimes}. 
  It is sufficient to
  prove that there is some uniform constant $\widetilde{K}>0$ such that
  if $A$ is a Borel set in $\{\theta\}\times\Y$ with
  diameter smaller than $\delta_1/2$ then
 $$
 \Leb\big((f_{\alpha^{-n}(\theta)}^{n})^{-1}(A)\cap
 Z(\alpha^{-n}(\theta),\lambda)\cap
 H_{n}(\al^{-n}(\te))\big)\big)\leq \widetilde{K} \Leb(A).
 $$
 Let $A$ be a Borel set in $\{\theta\}\times\Y$ with
 diameter smaller than $\delta_1/2$ and $B$ an open ball of
 radius $\delta_1/2$ containing $A$. We may write
 \begin{align*}
   (f_{\alpha^{-n}(\theta)}^{n})^{-1}(B)=\bigcup_{k\geq
     1}B_k,
 \end{align*}
 where $(B_k)_{k\geq 1}$ is a (possibly finite) family of
 two-by-two disjoint open sets in
 $\{\alpha^{-n}(\theta)\}\times\Y$. Discarding those $B_k$
 that do not intersect $Z(\alpha^{-n}(\theta),\lambda)\cap
 H_{n}(\al^{-n}(\te))$, we choose for each $k\geq 1$ a point
 $x_k\in Z(\alpha^{-n}(\theta),\lambda)\cap
 H_{n}(\al^{-n}(\te))\cap B_k$. 

 For $k\geq 1$ let $V_n(\alpha^{-n}(\theta),x_k)$ be the
 neighborhood of $x_k$ in $\{\alpha^{-n}(\theta)\}\times\Y$
 given by Proposition~\ref{pr:prophyptimes}. Since $B$ is
 contained in $B\big(f_{\alpha^{-n}(\theta)}^n(x_k),
 {\de_1}\big)$, the ball of radius $\de_1$ around
 $f_{\alpha^{-n}(\theta)}^n(x_k)$ in $\{\theta\}\times\Y$,
 and $f_{\alpha^{-n}(\theta)}^n$ is a diffeomorphism from
 $V_n(\alpha^{-n}(\theta),x_k)$ onto
 $B\big(f_{\alpha^{-n}(\theta)}^n(x_k), {\de_1}\big)$, we
 must have $B_k\subset V_n(\alpha^{-n}(\theta),x_k)$ (recall
 that by our choice of $B_k$ we have
 $f_{\alpha^{-n}(\theta)}^n(B_k)\subset B$). 

 As a consequence of this and item (3) of
 Proposition~\ref{pr:prophyptimes}, we have for every $k$
 that the map $f_{\alpha^{-n}(\theta)}^n\mid B_k\colon B_k\rightarrow B$
 is a diffeomorphism with bounded distortion:
 \begin{align*}
   \frac{1}{C_1} \leq \frac{|\det
     Df_{\alpha^{-n}(\theta)}^n(y)|} {|\det
     Df_{\alpha^{-n}(\theta)}^n(z)|} \leq C_1
   \quad\text{for all}\quad y,z\in B_k.
 \end{align*}
 This finally gives that $\Leb
 \big(f_{\alpha^{-n}(\theta)}^{-n}(A)\cap
 Z(\alpha^{-n}(\theta),\lambda)\cap
 H_n(\alpha^{-n}(\theta))\big)$ is bounded from above by
 \begin{align*}
   \sum_{k}\Leb\big(f_{\alpha^{-n}(\theta)}^{-n}(A\cap
   B)\cap B_k\big)
   \leq 
   \sum_{k}C_1\frac{\Leb(A\cap B)}{\Leb(B)}\Leb(B_k)
   \leq
   \widetilde{K} \Leb(A), 
\end{align*} 
where $\widetilde{K}>0$ is a constant 
 only depending on $C_1$, on the volume of the ball $B$ of
 radius $\delta_1/2$, and on the volume of $\Y$.
\end{proof}

The analogous statements to Lemmas~\ref{le:bdd-dens-mu}
and~\ref{le:bdd-dist-limit} are proved in the exact same
way.  At this point, we have the analogous results to
Corollary~\ref{MainCorollary} and Lemmas~\ref{le:bdd-dens},
~\ref{le:bdd-dens-mu} and~\ref{le:bdd-dist-limit}. The rest
of the argument proving the existence of absolutely
continuous invariant measures is entirely analogous. We also
obtain a similar statement to
Proposition~\ref{pr:existenceofacim}.

For the ergodic decomposition, the
  arguments are the same as in
  Section~\ref{sec:finitely-many-ergodi}, including a result
  analogous to Proposition~\ref{pr:finitelymanyergodic}
  whose proof is standard and follows \cite[Lemma
  5.6]{ABV00} using the bounded distortion property provided
  by item (3) of Proposition~\ref{pr:prophyptimes}.

\subsection{Non-invertible base map with higher-dimensional
  fibers}
\label{sec:case-with-higher-dim-fibers}

With the notation introduced in Section~\ref{sec:Remove-H2},
we define the map $\hat{\fhi}:\hat{\toro}\times \Y\ra
\hat{\toro}\times \Y$,
$\hat{\fhi}(\hat{\te},x)=(\hat{\al}(\hat{\te}),
\hat{f}(\hat{\te},x))$, where
$\hat{f}(\hat{\te},x)=f(\te_{0},x)$. In the exact same
manner as in Section~\ref{sec:Remove-H2}, we deduce that
this map satisfies conditions $(H_1), (H_2^*)$ and $(H_3)$,
if $\varphi$ satisfies conditions $(H_1)$ through $(H_3)$.

Moreover the argument about relative compactness and the
proofs of Lemmas~\ref{le:tightness}
and~\ref{le:invariance-hat} need no change.
We are left to show that if $\vfi$ is non-uniformly
expanding along the fibers, then $\hat\vfi$ is likewise.
But this follows from
\begin{itemize}
\item the easy observation that $
  \hat\vfi^k(\hat\theta,x)=(\sigma^k(\hat\theta),f^k_{\theta_0}(x))
  $;
\item together with the fact that the full
  $\nu\times\Leb$-measure subset $W$ of $\toro\times\Y$
  satisfying the conditions (\ref{eq:pos-lyap-vert}) and
  (\ref{eq:slow-rec}) of non-uniform expansion and slow
  recurrence provides the set $\hat W=\pi^{-1}(W)$ which
  also has full $\hat\nu\times\Leb$-measure on
  $\hat\toro\times\Y$.
\end{itemize}
So the points $(\hat\theta,x)\in\hat W$ will satisfy
conditions (\ref{eq:pos-lyap-vert}) and
(\ref{eq:slow-rec}). Hence $\hat\varphi$ is non-uniformly
expanding along the fibers, with a bijection $\hat\alpha$ as
the base transformation.

We can now apply the same arguments of
Sections~\ref{subs:basic-invariant-measures} and
\ref{sec:acim} to $\hat\varphi$. So our main results also
hold if we replace condition $(H_2^*)$ by condition $(H_2)$.


\appendix

\section{Measurability}
\label{sec:measur}

Here we prove that the measures $\eta_n$ defined on Section
\ref{subs:basic-invariant-measures} together with the
measures $\mu_n$ defined on Section \ref{sec:acim} are
well-defined. We consider separately the case with one
dimensional fibers and the case with higher dimensional
fibers.


\subsection{The measures $\eta_n$ are well defined}
\label{sec:measur-eta_n}

By the Hahn Extension Theorem, 
it is enough to define the measures on rectangles $A\times J$ with
$A\in
\mathcal{B}_{\toro}$ and $J\in \mathcal{B}_{I_0}$. It easily
follows from
\begin{proposition} \label{measuresetan} Let $J\subset I_0$
  be a Borel set. For every $n\in\N$, the function $\toro\ni
  \te\mapsto \eta_n(\te)(J)$ is measurable.
\end{proposition}
\begin{proof}
  Let us fix a set $J\in\mathcal{B}_{I_0}$. To prove the
  measurability of $\te\mapsto\eta_n(\te)(J)$ it suffices to
  prove the measurability of the functions
$\te \mapsto \eta^{i}_{J}(\te):=(f_{\al^{-i}(\te)}^{i})_{*}
\leb(J)$,
for $i\in \N$. Let us define the following functions
\begin{equation*}
\begin{aligned}
\al^{-1}\times id:&\:\,\toro\times I_0 \ra  \toro\times I_0 \\ 
&(\te,x)  \mapsto (\al^{-1}(\te), x)  
\end{aligned}
\qquad
\begin{aligned}
\pi_{\toro}:&\toro\times I_0 \to  \toro \\ 
& \: (\te,x)  \mapsto  \te  
\end{aligned}
\qquad
\begin{aligned}
\pi_{I_0}:&\toro\times I_0 \to  I_0 \\ 
& \: (\te,x)  \mapsto  x  
\end{aligned}
\end{equation*}
and $\chi_J$ is the characteristic function of $J$. The
projection maps are clearly measurable, considering on
$\toro\times I_0$ the $\sigma$-algebra
$\mathcal{B}_{\toro}\times \mathcal{B}_{I_0}$. Since
compositions of measurable maps are measurable maps,
$\al^{-1}\times id (\te,x)=(\al^{-1}\circ
\pi_{\toro}(\te,x), \pi_{I_0}(\te,x))$ is also measurable.

With these notations, we have that $\eta^{i}_{J}(\te)=
\int_{I_0} \phi_i(\te,x)\,d\leb(x)$, where
$\phi_i:\toro\times I_0\ra \R$ is defined by
\begin{equation} \label{phii}
(\te,x) \mapsto \phi_i(\te,x):= 
\chi_J\circ \pi_{I_0}\circ \fhi^i \circ (\al^{-i}\times id)^{i}
(\te,x).
\end{equation}
Using Fubini\s s Theorem, the measurability of $\phi_i$
(considering the $\sigma$-algebra $\mathcal{B}_{\toro}\times
\mathcal{B}_{I_0}$) implies the measurability of $\te\mapsto
\eta^{i}_{J}(\te)$ (considering the $\sigma$-algebra
$\mathcal{B}_{\toro}$) . 
\end{proof}

\subsection{The measures $\mu_n$ are well defined}
\label{sec:measur-mu_n}

We assume the skew-product satisfies the property
$(H_4)$. The proof for the case of $(H_4^*)$ is entirely
analogous. It is enough to substitute $\crit$ by $\disc$.

As in the case of $\eta_n$, the well-definition of the
measures $\mu_n$ follows from Hahn Extension Theorem and the
following result which implies that these measures are
defined on the algebra of the rectangles.

\begin{proposition} \label{measuresmun} Let $J\subset I_0$
  be a borelian set. For every $n\in\N$, the function $\toro\ni
  \te\mapsto \mu_n(\te)(J)$ is measurable.
\end{proposition}

In the definition of the measures $\mu_n$ appear the sets
$H_j(\te,\constantB)$ ($j\in\N, \te\in\toro$). These sets depend on
the
maps $r_j(\te,x)$ and $l^{*}_{j}(\te,x):=
|f_{\te}^{j}(T^j(\te,x))|$. We study first the measurability
of these functions.

Let us recall the definition of the function $r_i$ (given in
Section \ref{sec:acim}). Given $i\in\N$ and a point $(\te,
x)\in\toro\times I_0$, we denote by $T_i(\te,x)$ the maximal
interval such that $f_{\te}^{j}(T_i(\te,x))\cap
\crit_{\al^j(\te)}=\emptyset$ for all
$j<i$. 
Thus $r_i(\te,x)$ denotes the minimum of the lengths of the
connected components of
$f_{\te}^{i}(T_i(\te,x)\setminus\{x\})$.

\begin{lemma} \label{rimeasurable} The maps $r_i:\toro\times
  I_0 \ra \R$ are measurable, for all $i\in\N$.
\end{lemma}

\begin{proof} 
  For fixed $\te\in\toro$, $x\mapsto
  r_i(\te,x)$ is a continuous function, since $f_\te^{i}$ (for
$\te\in\toro$, $i\in\N$) are piecewise continuous $C^3$
maps. Hence, by \cite[Lemma 9.2]{mackey52}, we conclude $r_i$ is
  measurable, if for fixed $x\in I_0$ the function
  $\te\mapsto r_i(\te,x)$ is measurable. We claim that this
  last condition is true. 
 To prove it, we write $r_i(\cdot, x)$ as a composition of measurable
maps. 

For $i\in\N$, let us define the set
\begin{equation*}
\crit^{i}=\bigcup_{j=0}^{i-1} \fhi^{-j}\crit \cup (\toro\times
\partial I_0) 
\end{equation*}
Given $(\te,x)\in\toro\times I_0$, the interval
$T_i(\te,x)=(a_i(\te,x),b_i(\te,x))$ can be defined in the
following way
\begin{align*}
a_i(\te,x)=\sup (E^{x-})_{\te}:= \sup \{ y\in I_0; (\te,y)\in E^{x-}
\} \\
b_i(\te,x)=\inf (E^{x+})_{\te}:= \inf \{ y\in I_0; (\te,y)\in E^{x+}
\} 
\end{align*}
where $E^{x-}= (\toro\times (-\infty,x]\cap I_0)\cap
\crit^{i}$ and $E^{x+}= (\toro\times [x,+\infty)\cap
I_0)\cap \crit^{i}$. The sets $E^{x-}$ and $E^{x+}$ are
measurable, since by hypotheses $(H_1)$, $\crit$ is
measurable. Then, for fixed $x\in I_0$, the measurability of
the functions $\te\mapsto a_i(\te,x)$ and $\te\mapsto
b_i(\te,x)$ follows from the next result.

\begin{claim} Let $E$ be a set in
  $\mathcal{B}_{\toro}\times \mathcal{B}_{I_0}$ and let
  $S:\toro\ra I_0, s:\toro\ra I_0$ be functions defined by
  $S(\te)=\sup E_{\te}=\sup \{y\in I_0; (\te,y)\in E\},
  s(\te)=\inf E_{\te}=\inf \{y\in I_0; (\te,y)\in E\}$. Then
  $S$ and $s$ are measurable maps.
\end{claim}
\begin{proof}
  We prove first for the map $S$. Let $b\in \R$ be a
  constant. We want to prove that $S^{-1}((b,+\infty))\in
  \mathcal{B}_{\toro}$.  First, let us suppose that $E$ is
  an open set on $\toro\times I_0$. Let $\te_0$ be any point
  in $S^{-1}((b,+\infty))$. Then there exists $y_0\in I_0$
  such that $y_0 >b$ and $(\te_0,y_0)\in E$. The openness of
  $E$ shows the existence of open sets $A\subset \toro$ and
  $B\subset I_0$ such that $(\te_0,y_0)\in A\times B\subset
  E$. Thus $A\subset S^{-1}((b,+\infty))$ and it shows that
  $S^{-1}((b,+\infty))$ is an open set.

  In the general case, given any measurable set $E$, let us
  consider the sets
\begin{equation*}
  B\left(E,\frac{1}{n}\right)
  =\left\{z\in \toro\times I_0; 
    \:\dist(z,w)<\frac{1}{n} \text{ for some } w\in E\right\}.
\end{equation*}
for $n\in\N$. We consider the functions $S_n(\te)=\sup
\{y\in I_0; (\te,y)\in B(E,1/n)\}$. These functions are
measurable by what we have proved. Since $S=\inf_{n\in\N}
S_n$, the measurability of $S$ follows.

\end{proof}

Using the measurability of $a_i(\te,x)$ and $b_i(\te,x)$ we conclude
the measurability of $\te\mapsto r_i(\te,x)$ (all for fixed $x\in
I_0$). It finishes the proof of Lemma \ref{rimeasurable}.
\end{proof}

Now, we want to prove the measurability of the maps
$l^{*}_{j}$. Let us consider a sequence of measurable
partitions $\ldots\subset\mathcal{P}_{n+1}\subset
\mathcal{P}_{n}\subset\ldots\subset \mathcal{P}_1$ of $I_0$
such that the norm of $\mathcal{P}_n$ is less than
$1/n$. Choose a point $x^{n}_i$ in each $P^{n}_i$ element of
$\mathcal{P}_n$ and define the functions
\begin{equation*}
l^{n}_{j}(\te,x):=  |f_{\te}^{j}(T^j(\te,x^{n}_i))| \text{ for all }
x\in P^{n}_i .
\end{equation*}
We also consider the map $l_j:=\liminf_{n\to\infty} l^{n}_{j}$. 
\begin{lemma} \label{ljmeasurable}
The maps $l_{j}:\toro\times I_0\to \R$ are measurable for all
$j\in\N$. 
\end{lemma}

\begin{proof}
For fixed $x\in I_0$, the maps $\te\to |f_{\te}^{j}(T^j(\te,x))|$ are
measurable, since 
\begin{equation*}
|f_{\te}^{j}(T^j(\te,x))|=|\pi_{I_0}\circ\fhi^{i}\circ (id, a_i(\cdot,
x)) (\te)- \pi_{I_0}\circ\fhi^{i}\circ (id, b_i(\cdot,x) (\te)| .
\end{equation*}
Therefore the maps $l^{n}_{j}$ are measurable. Obviously it implies
the measurability of maps $l_j$.
\end{proof}

\begin{proof}[Proof of Proposition \ref{measuresmun}]
By Lemma \ref{ljmeasurable}, the map $l_j$ is measurable and
$l_j(\te,x)=l^{*}_j(\te,x)$ if $r_j(\te,x)>0$. By Lemma
\ref{rimeasurable}, the sets $\mathcal{H}_i(\secondelta):=\{ z\in
\toro\times I_0; r_i(z)>\secondelta\}$ are measurable, for any
$\secondelta>0$. These facts imply that
$H_i(\secondelta)=\mathcal{H}_i(\secondelta)\cap
{(l^{*}_{j})}^{-1}(3\secondelta,\infty)$ is a measurable set.


Let us fix a set $J\in\mathcal{B}_{I_0}$. As on Proposition
\ref{measuresetan}, to prove the measurability of
$\te\mapsto\mu_n(\te)(J)$ it suffices to prove the
measurability of the functions
$\te \mapsto \mu^{i}_{J}(\te):= (f_{\al^{-i}(\te)}^{i})_{*} (\leb |
H_{i}(\al^{-i}(\te),\constantB)\cap Z(\al^{-i}(\te),\la))(J)$,   
for $i\in \N$. 
Now, we have that $\mu^{i}_{J}(\te)=
\int_{I_0} \phi_i(\te,x)\psi_i(\te,x)\: d\leb(x)$, where
$\phi_i,\psi_i:\toro\times I_0\ra \R$, $\phi_i$ are
respectively defined in (\ref{phii}) and
\begin{equation*}
(\te,x)\mapsto \psi_i(\te,x):= \bigchi_{H_i(\constantB)}\circ
(\al^{-1}\times
id)^{i}(\te,x) \cdot \bigchi_{Z(\la)}\circ (\al^{-1}\times
id)^{i}(\te,x)
\end{equation*}
Once again, using Fubini\s s Theorem, the measurability of
$(\te,x)\mapsto \phi_i(\te,x) \psi_i(\te,x)$ implies the
measurability of $\te \mapsto \mu^{i}_{J}(\te)$.
\end{proof}


\subsection{Higher-dimensional fibers}
\label{sec:case-with-higher}

\subsubsection{The measures $\eta_n$ are well defined}
\label{sec:high-measur-eta_n}

This case is precisely the same as the case with
one-dimensional fibers, so we have nothing to add.

\subsubsection{The measures $\mu_n$ are well defined}
\label{sec:measures-mu_n-are}

From the definition of $\mu_n$ in the higher dimensional
case, we see that it is enough to show that for every
$n\in\N$ and Borel set $S\subset\Y$ the function
$\toro\ni\theta\mapsto \mu_n(\theta)(S)$ is measurable. For
this it is enough to prove the following.

\begin{lemma}
  \label{le:mu_n-higher-measurable}
  The function $\toro\ni\theta\mapsto
  \Leb\big( \H_j(\alpha^{-j}(\theta)) \cap
  (f^{j}_{\alpha^j(\theta)})^{-1}(S)\big)$ is measurable for each
fixed
  $j\in\N$ and measurable $S\subset\Y$.
\end{lemma}

Analogously to the previous subsection, we consider the maps
\begin{equation*}
\begin{aligned}
\al^{-1}\times id:&\:\,\toro\times\Y \to  \toro\times\Y \\ 
&(\te,x)  \mapsto (\al^{-1}(\te), x)  
\end{aligned}
\qquad
\begin{aligned}
\pi_{\toro}:&\toro\times \Y \to  \toro \\ 
& \: (\te,x)  \mapsto  \te  
\end{aligned}
\qquad
\begin{aligned}
\pi_{I_0}:&\toro\times \Y \to \Y \\ 
& \: (\te,x)  \mapsto  x  
\end{aligned}
\end{equation*}
and $\chi_S$ the characteristic function of $S$. These
functions are all measurable with respect to the
corresponding Borel $\sigma$-algebras.  We consider also
$\chi_{\H_n}$ the characteristic function of
$\H_n(\sigma,\delta,b)$.

\begin{lemma}
  \label{le:H_n-measurable}
  The set $\H_n(\sigma,\delta,b)$ is a Borel subset of
$\toro\times\Y$.
\end{lemma}

\begin{proof}
  According to the definition of
  $(\sigma,\delta,b)$-hyperbolic time
  \begin{align*}
    \H_n(\sigma,\delta,b) = \{(\theta,x)\in\toro\times\Y:
    (\ref{eq:tempo-hip}) \text{ is true for } (\theta,x) \}
  \end{align*}
  is an intersection of at most finitely many sets of the
  form $\{(\theta,x)\in\toro\times\Y : g(\theta,x)>c\}$ for
  a measurable function $g:\toro\times\Y\to\R$ and some
  constant $c\in\R$. Indeed, if we define for
  $k=0,\dots,n-1$
  \begin{align*}
    g_k(\theta,x):=\prod_{j=n-k}^{n-1}\|Df_{\alpha^j(\theta)}(
    f^j_\theta(x))^{-1}\| \quad\text{and}\quad
    d_k(\theta,x):=\dist_\delta \big(
    f^k_\theta(x),\cc\cap(\{\alpha^k(\theta)\}\times\Y)
    \big),
  \end{align*}
  then we can write
  \begin{align*}
    \H_n(\sigma,\delta,b) = \{ (\theta,x) \in\toro\times\Y :
    g_k(\theta,x)<\sigma^k \quad\text{and}\quad
    d_k(\theta,x)> e^{-b k}, k=0,\dots,n-1\}.
  \end{align*}
  Thus $\H_n(\sigma,\delta,b)$ is a Borel subset of
  $\toro\times\Y$ as soon as we show that $g_k,d_k$ are
  measurable functions for each $k\ge0$.

  Clearly $g_k$ is measurable from condition $(H_6)$. For
  the functions $d_k:\toro\times\Y\to[0,+\infty)$ we clearly
  have
  \begin{align*}
    d_k(\theta,x)=D(\alpha^k(\theta),f^k_\theta(x))
    \quad\text{where}\quad D(\theta,x)=\inf\xi_{(\theta,x)}
  \end{align*}
  and we define
  \begin{align*}
    \xi(\theta,x,y)=
    \xi_{(\theta,x)}(y):=\dist_\delta(x,y)\cdot\chi_{\cc}(\theta,y)+
    \delta\cdot(1-\chi_{\cc}(\theta,y)).
  \end{align*}
  Clearly $\xi:\toro\times\Y\times\Y\to[0,\delta]$ is
  measurable, so $D:\toro\times\Y\to[0,\delta]$ is also
  measurable and $d_k$ is a composition of $D$ with other
  measurable maps from condition $(H_5)$. This completes the
  argument showing that $\H_n(\sigma,\delta,b)$ is a Borel
  subset of $\toro\times\Y$.
\end{proof}

Now we are ready to prove the first lemma.

\begin{proof}[Proof of Lemma~\ref{le:mu_n-higher-measurable}.]
  We note that we can write
  \begin{align}\label{eq:lambda-higher-dim}
    \Leb\big( \H_j(\alpha^{-j}(\theta)) \cap
    (f^{j}_{\alpha^j(\theta)})^{-1}(S)\big) = \int
    \phi_j(\theta,x)\psi_j(\theta,x) \,d\Leb(x),
  \end{align}
  where
  \begin{align*}
    \phi_j(\theta,x) := \chi_S\circ \pi_{I_0} \circ \fhi^j
    \circ (\al^{-j}\times id)^{j} (\te,x) \quad\text{ and
    }\quad \psi_j(\theta,x) := \chi_{\H_n}\circ
    (\al^{-j}\times id)^{j}(\te,x) .
  \end{align*}
  Since both $\phi_j$ and $\psi_j$ are Borel measurable from
  $\toro\times\Y$ to $\R$, Fubini's Theorem ensures
  that~(\ref{eq:lambda-higher-dim}) is a measurable function
  of $\theta\in\toro$, as we need. This concludes the proof.
\end{proof}

With Lemma~\ref{le:mu_n-higher-measurable} we complete the
proof of the measurability of all functions used in the
previous sections.


\def\cprime{$'$}

 \bibliographystyle{abbrv}

\begin{thebibliography}{10}

\bibitem{Adl96}
K.~Adl-Zarabi.
\newblock Absolutely continuous invariant measures for piecewise
expanding
  {$C^2$} transformations in {${\bf R}^n$} on domains with cusps on
the
  boundaries.
\newblock {\em Ergodic Theory Dynam. Systems}, 16(1):1--18, 1996.

\bibitem{AlSchn}
J.~Alves and D.~Schnellmann.
\newblock Ergodic properties of viana-like maps with singularities in
the base
  dynamics.
\newblock {\em to appear in Proc. Amer. Math. Soc.}, 2012.

\bibitem{Al00}
J.~F. Alves.
\newblock {SRB measures for non-hyperbolic systems with
multidimensional
  expansion}.
\newblock {\em {Ann. Sci. {\'E}cole Norm. Sup.}}, {33}:{1--32},
{2000}.

\bibitem{AA03}
J.~F. Alves and V.~Araujo.
\newblock {Random perturbations of nonuniformly expanding maps}.
\newblock {\em {Ast{\'e}risque}}, {286}:{25--62}, {2003}.

\bibitem{ABV00}
J.~F. Alves, C.~Bonatti, and M.~Viana.
\newblock {SRB measures for partially hyperbolic systems whose central
  direction is mostly expanding}.
\newblock {\em {Invent. Math.}}, {140}({2}):{351--398}, {2000}.

\bibitem{alves-viana2002}
J.~F. Alves and M.~Viana.
\newblock {Statistical stability for robust classes of maps with
non-uniform
  expansion.}
\newblock {\em {Ergodic Theory and Dynamical Systems}}, {22}:{1--32},
{2002}.

\bibitem{ArTah}
V.~Ara{\'u}jo and A.~Tahzibi.
\newblock {Stochastic stability at the boundary of expanding maps.}
\newblock {\em {Nonlinearity}}, {18}:{939--959}, {2005}.

\bibitem{arnold-l-1998}
L.~Arnold.
\newblock {\em {Random dynamical systems}}.
\newblock {Springer-Verlag}, {Berlin}, {1998}.

\bibitem{billingsley99}
P.~Billingsley.
\newblock {\em Convergence of probability measures}.
\newblock Wiley Series in Probability and Statistics: Probability and
  Statistics. John Wiley \& Sons Inc., New York, second edition, 1999.
\newblock A Wiley-Interscience Publication.

\bibitem{Buzz00}
J.~Buzzi.
\newblock Absolutely continuous {S}.{R}.{B}. measures for random
  {L}asota-{Y}orke maps.
\newblock {\em Trans. Amer. Math. Soc.}, 352(7):3289--3303, 2000.

\bibitem{Bu00}
J.~Buzzi.
\newblock {A.c.i.m{'}s for arbitrary expanding piecewise real-analytic
mappings
  of the plane}.
\newblock {\em {Ergod. Th. \& Dynam. Sys.}}, {20}:{697--708}, {2000}.

\bibitem{buzzi-sester-tsujii}
J.~Buzzi, O.~Sester, and M.~Tsujii.
\newblock {Weakly expanding skew-products of quadratic maps}.
\newblock {\em {Ergodic Theory Dynam. Systems}},
{23}({5}):{1401--1414},
  {2003}.

\bibitem{CoFoSi82}
I.~P. Cornfeld, S.~V. Fomin, and Y.~G. Sina\u\i.
\newblock {\em {Ergodic theory}}, volume {245} of {\em {Grundlehren
der
  Mathematischen Wissenschaften [Fundamental Principles of
Mathematical
  Sciences]}}.
\newblock {Springer-Verlag}, {New York}, {1982}.
\newblock {Translated from the Russian by A. B. Sosinski\u\i}.

\bibitem{MS93}
W.~{de Melo} and S.~{van Strien}.
\newblock {\em {One-dimensional dynamics}}.
\newblock {Springer Verlag}, {1993}.

\bibitem{DenkGord99}
M.~Denker and M.~Gordin.
\newblock Gibbs measures for fibred systems.
\newblock {\em Adv. Math.}, 148(2):161--192, 1999.

\bibitem{DenkGordHein02}
M.~Denker, M.~Gordin, and S.-M. Heinemann.
\newblock On the relative variational principle for fibre expanding
maps.
\newblock {\em Ergodic Theory Dynam. Systems}, 22(3):757--782, 2002.

\bibitem{GorBoy89}
P.~G{\'o}ra and A.~Boyarsky.
\newblock Absolutely continuous invariant measures for piecewise
expanding
  {$C^2$} transformation in {${\bf R}^N$}.
\newblock {\em Israel J. Math.}, 67(3):272--286, 1989.

\bibitem{gouezel07}
S.~Gou{\"e}zel.
\newblock Statistical properties of a skew product with a curve of
neutral
  points.
\newblock {\em Ergodic Theory Dynam. Systems}, 27(1):123--151, 2007.

\bibitem{kasriel}
R.~H. Kasriel.
\newblock {\em Undergraduate topology}.
\newblock Dover Publications Inc., Mineola, NY, 2009.
\newblock Reprint of the 1971 original [MR0283741].

\bibitem{KH95}
A.~Katok and B.~Hasselblatt.
\newblock {\em {Introduction to the modern theory of dynamical
systems}},
  volume~{54} of {\em {Encyclopeadia Appl. Math.}}
\newblock {Cambridge University Press}, {Cambridge}, {1995}.

\bibitem{Ke79}
G.~Keller.
\newblock {Ergodicit{\'e} et mesures invariantes pour les
transformations
  dilatantes par morceaux d{'}une r{\'e}gion born{\'e}e du plan}.
\newblock {\em {C.R. Acad. Sci. Paris}}, {A 289}:{625--627}, {1979}.

\bibitem{Ke90}
G.~Keller.
\newblock {Exponents, attractors and Hopf decompositions}.
\newblock {\em {Ergod. Th. \& Dynam. Sys.}}, {10}:{717--744}, {1990}.

\bibitem{mackey52}
G.~W. Mackey.
\newblock Induced representations of locally compact groups. {I}.
\newblock {\em Ann. of Math. (2)}, 55:101--139, 1952.

\bibitem{Man87}
R.~Ma{\~n}{\'e}.
\newblock {\em {Ergodic theory and differentiable dynamics}}.
\newblock {Springer Verlag}, {New York}, {1987}.

\bibitem{Mor85}
T.~Morita.
\newblock Random iteration of one-dimensional transformations.
\newblock {\em Osaka J. Math.}, 22(3):489--518, 1985.

\bibitem{Pelik84}
S.~Pelikan.
\newblock Invariant densities for random maps of the interval.
\newblock {\em Trans. Amer. Math. Soc.}, 281(2):813--825, 1984.

\bibitem{Pinheiro05}
V.~Pinheiro.
\newblock {Sinai-Ruelle-Bowen measures for weakly expanding maps}.
\newblock {\em {Nonlinearity}}, {19}({5}):{1185--1200}, {2006}.

\bibitem{Sa00}
B.~Saussol.
\newblock {Absolutely continuous invariant measures for
multi-dimensional
  expanding maps}.
\newblock {\em {Israel J. Math}}, {116}:{223--248}, {2000}.

\bibitem{schnell08}
D.~Schnellmann.
\newblock Non-continuous weakly expanding skew-products of quadratic
maps with
  two positive {L}yapunov exponents.
\newblock {\em Ergodic Theory Dynam. Systems}, 28(1):245--266, 2008.

\bibitem{schnell09}
D.~Schnellmann.
\newblock Positive {L}yapunov exponents for quadratic skew-products
over a
  {M}isiurewicz-{T}hurston map.
\newblock {\em Nonlinearity}, 22(11):2681--2695, 2009.

\bibitem{solano}
J.~Solano.
\newblock Non-uniform hyperbolicity and existence of absolutely
continuous
  invariant measures.
\newblock {\em Bull. of the Braz. Math. Soc.}, 44:67-103, 2013.

\bibitem{Tsu05}
M.~Tsujii.
\newblock Physical measures for partially hyperbolic surface
endomorphisms.
\newblock {\em Acta Math.}, 194(1):37--132, 2005.

\bibitem{Vi97}
M.~Viana.
\newblock {Multidimensional nonhyperbolic attractors}.
\newblock {\em {Inst. Hautes {\'E}tudes Sci. Publ. Math.}},
{85}:{63--96},
  {1997}.

\end{thebibliography}

\end{document}